\documentclass[11pt]{article}

\usepackage[a4paper,margin = 1in]{geometry}

\usepackage{amsmath,amssymb,bbm,mathrsfs,amsbsy} 

\usepackage{amsthm}

\usepackage{multicol}
\usepackage{multirow}
\usepackage{diagbox}

\usepackage[labelfont=bf]{caption}

\usepackage{comment}

\usepackage{setspace}
\usepackage{stix}

\usepackage[shortlabels]{enumitem}

\usepackage{tabularx}

\usepackage{tikz}

\usepackage[hidelinks]{hyperref}

\usepackage{enumitem}

\usepackage{parskip}

\usepackage{url}
\makeatletter
\g@addto@macro{\UrlBreaks}{\UrlOrds}
\makeatother

\usepackage[
backend=biber,
sorting=nty
]{biblatex}

\newtheorem{theorem}{Theorem}[section]
\newtheorem{corollary}[theorem]{Corollary}
\newtheorem{lemma}[theorem]{Lemma}
\newtheorem{conjecture}[theorem]{Conjecture}
\newtheorem{proposition}[theorem]{Proposition}
\newtheorem{definition}{Definition}[section]
\newtheorem{remark}{Remark}
\newtheorem{example}{Example}

\numberwithin{equation}{section}
\usepackage{chngcntr}
\counterwithin{figure}{section}
\counterwithin{table}{section}

\usepackage{appendix}

\usepackage{mathtools} 

\usepackage{authblk}

\newcommand{\A}{\mathscr{A}}
\newcommand{\B}{\mathscr{B}}

\newcommand{\F}{\mathbb{F}}

\newcommand{\Pen}{\mathscr{P}}
\newcommand{\Prj}{\boldsymbol{P}}

\newcommand{\resultant}{\text{resultant}}
\newcommand{\Prb}{\mathbb{P}}
\newcommand{\chr}{\mathrm{char} \,}

\newcommand{\abs}[1]{\left\lvert #1 \right\rvert}
\DeclarePairedDelimiter\autopr{(}{)}
\newcommand{\pr}[1]{\autopr*{#1}}

\title{\textbf{Poncelet Triangles and Tetragons over Finite Fields}}
\author[1,2]{\textbf{Milena Radnovi\'c}\thanks{Email:milena.radnovic@sydney.edu.au}}
\author[1]{\textbf{Ruzzel Ragas}\thanks{Email: ruzzel.ragas@sydney.edu.au}}
\affil[1]{\textit{The University of Sydney, School of Mathematics and Statistics, Carslaw F07, 2006 NSW, Australia}}
\affil[2]{\textit{Mathematical Institute SANU, Belgrade, Kneza Mihaila 36, 11000 Belgrade, Serbia}}
\date{}

\begin{document}

\maketitle

\begin{quote}
\textbf{Abstract:} In the projective plane over a finite field of characteristic not equal to 2, we compute the probability that a randomly selected pair of distinct conics $(\A,\B)$, with $\A$ smooth or singular and $\B$ smooth, in a fixed pencil of conics will admit a triangle or a tetragon inscribed in $\A$ and circumscribed about $\B$. We do this for all pencils, classified up to projective automorphism, with at least one smooth conic; effectively allowing the case where our conic pairs intersect non-transversally.
\end{quote}

\tableofcontents

\section{Introduction}
\label{sec:intro}

Quoted as one of the most beautiful results in projective geometry, Poncelet's theorem has been discussed extensively, together with its vast connection to other areas of mathematics; see, for example \cite{drarad2011poncelet, fla2008poncelet, grihar1977poncelet}. As a brief recall, Poncelet's theorem states that if we have an $n$-sided polygon that is inscribed in a conic $\A$ and circumscribed about another conic $\B$, then we can construct another $n$-sided polygon that is inscribed and circumscribed about the same pair of conics, starting from any point in $\A$ as a vertex. This holds for a projective plane over any field of characteristic not equal to 2 \cite{ber1987geom, jowlia2023poncelet}. In this context, we call such polygon a Poncelet $n$-gon and say that the ordered pair $(\A,\B)$ is an $n$-Poncelet pair. 

In the setting of the projective plane over a finite field of characteristic not equal to 2, the probability of obtaining a $3$-Poncelet pair was discussed in \cite{chi2017ptc}. In this setting, Cayley's condition \cite{grihar1978cayley} and Dickson's classification  of pencils of conics \cite{dic1908penauto} have been utilized to compute bounds for the probability of obtaining a $3$-Poncelet pair among smooth pairs of conics, for each pencil with transversally intersecting generators. Under the same setting, \cite{wan2025pnc} gave exact probabilities of obtaining $3$-Poncelet pairs, and provided bounds for the probabilities of obtaining $4$-Poncelet pairs.
 
We develop on these previous works by considering a larger collection of conic pairs in the projective plane over a finite field of characteristic not equal to 2. In this paper, we allow the conic $\A$ in which the Poncelet polygon is inscribed to be smooth or singular while retaining the smoothness condition for the conic $\B$ about which the Poncelet polygon is circumscribed. With this, we consider every pencil in Dickson's classification that will contain at least one smooth conic, effectively allowing the case where our conic pairs intersect non-transversally. In this setting, we focus on Poncelet triangles and tetragons, and we aim to compute the exact probability of obtaining a $3$-Poncelet pair or $4$-Poncelet pair for each pencil under consideration. Moreover, we discuss in detail the probabilities for $3$-Poncelet pairs in the instance where the characteristic of the finite field is 3.

This paper is organised as follows.
In Section \ref{sec:pon_cay}, we review pencils of conics and Cayley's condition. 
In Section \ref{sec:prob}, we define the collection of conic pairs that we consider in this paper and describe precisely how we obtain a random pair of conics. 
Then, within this framework, we derive the probabilities of obtaining a $3$-Poncelet pair or $4$-Poncelet pair. 
We conclude with final remarks and possible future directions in Section \ref{sec:conclusion}. 
Some details about the classification of pencils of conics over finite fields are provided in Appendix \ref{sec:dicksonclass}.

\subsection{Notation}

We denote by $\F_q$ a finite field with $q$ elements and refer to $q$ as its order. Such a field exists if and only if $q = p^k$ where $k$ is a positive integer and $p$ is prime, which we refer to as the characteristic of our finite field, denoted by char $\F_q$. In this paper, we always work on the base plane $\Prj^2(\F_q)$ with char~$\F_q \neq 2$. 

Conics will be denoted by capital script letters $\A$ and $\B$ with their matrix representation denoted by $A$ and $B$, respectively. Following the notations in \cite{chi2017ptc}, we denote the algebraic closure of $\F_q$ as $\overline{\F_q}$ and given a conic $\A$ in $\Prj^2(\F_q)$, we define the conic $\overline{\A}$ in $\Prj^2(\overline{\F_q})$ as the conic defined by the same equation as $\A$. 

A pencil of conics generated by $\A$ and $\B$ is the collection of conics with matrix representation in $\{ \eta A~+~B \mid\eta \in \Prj^1(\F_q) \}$. 
We follow the approach in \cite{chi2017ptc} by considering pencils of conics up to projective automorphisms using Dickson's classification \cite{dic1908penauto}, which we give in Table \ref{tab:pencil_class_auto} in Appendix \ref{sec:dicksonclass}, see also \cite{hir1998projfinite}. 
For any pencil $\Pen_j$ from Table \ref{tab:pencil_class_auto}, each parameter $\eta \in \Prj^1(\F_q)$ corresponds to the conic with matrix representation $C_j(\eta) = \eta A_j + B_j$ where $A_j$ and $B_j$ are the matrix representations of the generators $\A_j$ and $\B_j$, with the convention that $C_j(\infty) = A_j$. For a comprehensive discussion about projective geometry, see for example, \cite{ric2011projgeom} and \cite{dumfoo2004absalgebra} for that of finite fields.

In this paper, we consider all pencils with at least one smooth element, leaving us with 11 pencils to consider. See Table \ref{tab:pencil_sing_base} in Appendix \ref{sec:dicksonclass} for important properties of these pencils, such as the index $\eta \in \Prj^1(\F_q)$ corresponding to singular elements of the pencil and the coordinates of the base points.

\section{Poncelet's Theorem and Cayley's Condition}
\label{sec:pon_cay}

Now, we state Poncelet's theorem as used in our setting.
\begin{theorem}[Poncelet \cite{ber1987geom, jowlia2023poncelet, pon1822poncelet}] 
    Let $(\A,\B)$ be a pair of conics in $\Prj^2(\F_q)$. If there exists an $n$-sided polygon that is inscribed in $\A$ and circumscribed about $\B$, then every point in $\A$ is a vertex of an $n$-sided polygon that is inscribed in $\A$ and circumscribed about $\B$.
\end{theorem}

We note that the Poncelet theorem in finite projective planes was studied in \cite{hk2018}.

Poncelet's theorem tells us that the number of sides of a polygon inscribed in the conic $\A$ and circumscribed about $\B$ is only dependent on the pair $(\A,\B)$. 
Thus we can ask whether a pair $(\A,\B)$ is an $n$-Poncelet pair, which can be answered using Cayley's condition.
Here we state that condition for the triangle and tetragon:

\begin{theorem}[Cayley \cite{cay1854cayley, grihar1978cayley}] 
    \label{thm:cayley}
    Consider the formal series expansion
    \begin{equation*}
        \sqrt{\det(tA+B)} = H_0 + H_1 t + H_2 t^2 + H_3 t^3 + \cdots
    \end{equation*}
    where $A$ and $B$ are the matrix representations for conics $\A$ and $\B$, respectively.

    Then, 
    \begin{itemize}
        \item $(\A,\B)$ is a $3$-Poncelet pair if and only if $H_2 = 0$.
        \item $(\A,\B)$ is a $4$-Poncelet pair if and only if $H_3 = 0$.
    \end{itemize}
\end{theorem}

More precisely, our statement of Poncelet's theorem and Cayley's condition works for $(\overline{\A},\overline{\B})$ in $\Prj^2(\overline{\F_q})$. In this paper, since our condition only depends on the matrix representations $A$ and $B$, we say that $(\A,\B)$ is an $n$-Poncelet pair in $\Prj^2(\F_q)$ if the induced pair $(\overline{\A},\overline{\B})$ in $\Prj^2(\overline{\F_q})$ is an $n$-Poncelet pair. Geometrically, this is equivalent to having at least one of the vertices of our Poncelet $n$-gon in $\A$. 

If $\A$, $\B$ correspond to the parameters $r$, $s$ of a given pencil, the conditions $H_2=0$ and $H_3=0$ in Theorem \ref{thm:cayley} will take the form of a polynomial equality in $r$ and $s$.
This will be explored and used in the next section.

\section{Computation of Probabilities}
\label{sec:prob}

In this section, we define precisely the probability spaces under consideration and derive the probabilities of obtaining a $3$-Poncelet pair or a $4$-Poncelet pair.

First, we need to restrict our consideration to the pairs $(\A,\B)$ of conics which allow the construction of Poncelet polygons.
First, the tangent lines are well defined only for non-singular conics, thus $\B$ must be smooth.
We can inscribe a polygon in a conic which is the union of two lines, so that the vertices alternate between them \cite{drarad2025poncelet}.
If $\A$ is singular, we will exclude the following cases:
\begin{itemize}
	\item $\A$ is a double line. Then every Poncelet polygon inscribed in $\A$ has a single vertex. Indeed, each tangent line to $\B$ passing through the starting vertex intersects $\A$ only at that point, treated as an intersection of multiplicity 2.

\item 
$\A$ is the union of two lines, with one them being tangent to $\B$.
In this case, we eventually obtain a vertex lying in a line tangent to $\B$ which coincides with one of the lines making up $\A$, when the construction terminates, since the next vertex is no longer well-defined.
\end{itemize}

\begin{definition}
A pair of conics $(\A,\B)$ is \emph{valid} if those two conics are distinct, $\B$ is smooth and $\A$ is either smooth or the union of two distinct lines which are not tangent to $\B$.
\end{definition}

\begin{remark}
The following pencils from Table \ref{tab:pencil_sing_base} 
contain a double line, which corresponds to the parameter $r$:
	\begin{itemize}
		\item in $\Pen_6$ for $r=0$;
		\item in $\Pen_8$ and $\Pen_{17}$ for $r=\infty$. 
	\end{itemize}
In the following pencils from that table, a pair of conics $(\A,\B)$, corresponding to the parameters $(r,s)$ is such that $\A$ is the union of two lines, one of which is tangent to $\B$: 	
\begin{itemize}
\item in $\Pen_4$, $\Pen_6$, $\Pen_{15}$, for $r=\infty$, $s\neq0$;
\item in $\Pen_5$, for $r=\infty$ and $s\neq\infty$;
\item in $\Pen_{17}$, for $r=0$ and $s\neq0$. 
	\end{itemize}
\end{remark}

\begin{definition}
In the pencil $\Pen_j$, denote by $\Phi_j$ the set of all valid pairs.
and by $\Psi_j$ its subset of all smooth pairs of conics.
\end{definition}

Note that each of the sets $\Phi_j$, $\Psi_j$ is non-empty for $j\in\{3,4,5,6,8,14,15,16,17,18,19\}$, except for $\Psi_3$ when $q=3$. 
In \cite{chi2017ptc, wan2025pnc}, the sample spaces $\Psi_j$ were considered, with the uniform measure.
In this paper, the sample space is extended to $\Phi_j$.
While other variants for choosing a probability measure are conceivable, in order to allow comparison with the existing results \cite{chi2017ptc, wan2025pnc}, we also adopt here the uniform measure.

\begin{example}
    As an illustration, we consider the pencil $\Pen_3$ in the projective plane over $\F_7$.
  For each valid pair in that pencil, we summarize in Table \ref{tab:ndist_pen3_f7} the values of $n$ such that the pair $(r,s)$ corresponds to an $n$-Poncelet pair. 
  Giving each valid pair in pencil $\Pen_3$ an equal chance of being selected, the probability of selecting a $3$-Poncelet pair is equal to the proportion of $n=3$, which is $\frac{2}{35}$. Similarly, the probabilities of selecting an $n$-Poncelet pair for $n=4$, $n=6$, $n=8$ are respectively $\frac{3}{7}$, $\frac{12}{35}$, $\frac{6}{35}$. For all other values of $n$, there are no $n$-Poncelet pairs in this pencil, so the probability equals zero for $n\not\in\{3,4,6,8\}$.
  
  On the other hand, restricting to the subset $\Psi_j$ of all smooth pairs in $\Pen_3$, the probabilities for a randomly selected pair being an $n$-Poncelet pair equal $\frac1{10}$, $\frac35$, $\frac3{10}$ for $n=3$, $n=4$, $n=6$ respectively and $0$ for all other values of $n$.
\end{example}

\begin{table}[!htb]
\centering
\caption{$n$-Poncelet pairs in $\Prj^2(\F_7)$. 
${}^{\dag}$ denotes singular elements of the pencil.}
\begin{tabular}{|l|c|c|c|c|c|c|c|c|}
\hline
\diagbox{r}{s}  & \multicolumn{1}{l|}{0 ${}^{\dag}$} & \multicolumn{1}{l|}{1 ${}^{\dag}$} & \multicolumn{1}{l|}{2} & \multicolumn{1}{l|}{3} & \multicolumn{1}{l|}{4} & \multicolumn{1}{l|}{5} & \multicolumn{1}{l|}{6} & \multicolumn{1}{l|}{$\infty {}^{\dag}$} \\ \hline
0  ${}^{\dag}$      &                        &                        & {4}                     & 8                      & 6                      & 8                      & 6                      &                               \\ \hline
1  ${}^{\dag}$      &                        &                        & 6                      & 8                      & 6                      & 8                      & {4}                      &                               \\ \hline
2        &                        &                        &                        & {4}                      & {4}                      & {4}                      & {4}                      &                               \\ \hline
3        &                        &                        & 6                      &                        & 6                      & {3}                      & 6                      &                               \\ \hline
4       &                        &                        & {4}                      & {4}                      &                        & {4}                      & {4}                      &                               \\ \hline
5        &                        &                        & 6                      & {3}                      & 6                      &                        & 6                      &                               \\ \hline
6        &                        &                        & {4}                      & {4}                      & {4}                      & {4}                      &                        &                               \\ \hline
$\infty {}^{\dag}$ &                        &                        & 6                      & 8                      & {4}                      & 8                      & 6                      &                               \\ \hline
\end{tabular}
\label{tab:ndist_pen3_f7}
\end{table}

\subsection{Probabilities for Poncelet Triangle}

 We first derive the form of $H_2$ from Cayley's condition in terms of $r$ and $s$ for each pencil $\Pen_j$. We use the notation $H_{2,j}$ to refer to the form of $H_2$ for pencil $\Pen_j$ and summarize them in Table \ref{tab:h2_summary}.

\begin{table}[!htb]
\centering
\renewcommand{\arraystretch}{1.6}%
\caption{Summary of Cayley's condition for the triangle case}
\begin{tabular}{c|l}
    \hline
    $\Pen_j$ & $H_{2,j}{}^{\dag}$ \\
    \hline
    $\Pen_{3}$ &  $H_{2,3} (r,s) = r^2 + (-4s^3 + 6s^2 - 4s)r + s^4$ \\
    &  $H_{2,3}(\infty,s) = u {}^{\ddag}$ \\
    \hline
    $\Pen_{4}$ & $ H_{2,4} (r,s) = 4r-s$ \\
    & $H_{2,4}(\infty,s) = 0$ \\
    \hline
    $\Pen_{5}$ & $ H_{2,5} (r,s) = 3$ \\
    & $H_{2,5}(\infty,s) = 0$ \\
    \hline
    $\Pen_{6}$ & $ H_{2,6} (r,s) = 4r-s$ \\
    & $H_{2,6}(\infty,s) = 0$ \\
    \hline
    $\Pen_{8}$ & $ H_{2,8} (r,s) = 3$ \\
    & $H_{2,8}\infty,s) = 0$ \\
    \hline 
    $\Pen_{14}$ & $ H_{2,14} (r,s) = (1-4e)r^2 + (-4 e^2 s^3 + 6 e s^2 + 4 e s - 4 s + 2)r + (e^2 s^4 - 6 e s^2 + 4 s - 3)$ \\
    & $H_{2,14}(\infty,s) = u {}^{\ddag}$ \\
    \hline
    $\Pen_{15}$ & $ H_{2,15} (r,s) = 4r-s$ \\
    & $H_{2,15}(\infty,s) = 0$ \\
    \hline
    $\Pen_{16}$ & $ \begin{aligned}
        H_{2,16} (r,s) = \, & [(1-4d)(1-4e)] r^2  \\
        & + [- 4 s^3 + 6 s^2 + (-16 d e + 4 d + 4 e - 4)s + (-8 d e + 2 d + 2 e)] r \\
        &+ [s^4 + (24 d e - 6 d - 6 e) s^2 + (-16 d e + 4 d + 4 e)s \\
        &+ (-48 d^2 e^2 + 24 d^2 e + 24 d e^2 - 3 d^2 - 6 d e - 3 e^2)]
    \end{aligned}$ \\
    & $H_{2,16}(\infty,s) = u {}^{\ddag}$ \\
    \hline
    $\Pen_{17}$ & $ H_{2,17}(r,s) = r (r - 4s)$ \\
    & $H_{2,17}(\infty,s) = u {}^{\ddag}$ \\
    \hline 
    $\Pen_{18}$ & $\begin{aligned}
        H_{2,18} (r,s) = \, & [3 s^4 + 4bs^3 +6cs^2 + 12s + (4b-c^2)] r^2  \\
        & + [2 b s^4 + (4 b^2 - 4 c) s^3 + (6 b c - 18) s^2 + (-4 b + 4 c^2) s + 2 c] r \\
        &+ [(-b^2 + 4 c) s^4 + 12 s^3 + 6 b s^2 + 4 c s + 3]
    \end{aligned}$ \\
    & $H_{2,18}(r,\infty) =  3 r^2  + 2 b r + (- b^2 + 4 c)$ \\
    & $H_{2,18}(\infty,s) =  3 s^4 + 4bs^3 +6cs^2 + 12s + (4b-c^2)$ \\
    \hline
    $\Pen_{19}$ & $  \begin{aligned}
        H_{2,19} (r,s) = \, & [\rho^2 - 4\nu \sigma^2] r^2  \\
        & + [-4 \nu^2 s^3 - 6 \nu\rho s^2  + (4 \nu \sigma^2 - 4 \rho^2) s - 2 \rho \sigma^2] r \\
        &+ [\nu^2 s^4 - 6 \nu \sigma^2 s^2  - 4 \rho \sigma^2 s  - 3 \sigma^4]
    \end{aligned}$ \\
    & $H_{2,19}(\infty,s) = u {}^{\ddag}$ \\
    \hline 
\end{tabular}

\label{tab:h2_summary}

\flushleft

\footnotesize{$^\dag$ $H_{2,j} = H_2$ from Cayley's condition for pencil $\Pen_j$ where non-zero scalar multiples and denominators are ignored}\\
\footnotesize{$^\ddag$ $u$ is a non-zero constant for char $\F_q \neq 2$} \\

\end{table}

With the aid of Table \ref{tab:h2_summary}, we can now count the number of valid pairs in $\Pen_j$ that correspond to a $3$-Poncelet pair by counting the number of pairs $(r,s)$ such that $H_{2,j}(r,s) = 0$. Since we only consider distinct conic pairs, we have Lemma \ref{lem:h2_ss_0} based on \cite{wan2025pnc}, which will help us adjust our counts for the case where $r=s$.
\begin{lemma}
    \label{lem:h2_ss_0}
    Let $H_{2,j}$ be the form of $H_2$ under $\Pen_j$ and $C_j(s)$ be the matrix representation of the element corresponding to $s \in \Prj^1(\F_q)$ in pencil $\Pen_j$. Then,
    \begin{enumerate}
        \item $H_{2,j}(s,s) = 0$ if and only if char $\F_q = 3$ or $C_j(s)$ is singular.
        \item if $C_j(r)$ is singular then $H_{2,j}(r,s) = 0$ if and only if $r=s$.
    \end{enumerate}
\end{lemma}

\begin{proof}
    Let $A = B = C_j(s)$, then $\sqrt{det(tA+B)} = \sqrt{(t+1)^3 \det(C_j(s))}$ which yields $H_{2,j}(s,s) = \frac{3}{8}\sqrt{\det(C_j(s))}$ which is zero if and only if char $\F_q = 3$, or $A = C_j(s)$ is singular.

    Now, suppose $C_j(r)$ is singular, we can use Table \ref{tab:h2_summary} to substitute every value of $r$ associated with a singular element as summarized in Table \ref{tab:pencil_sing_base}. We need to do this for all pencils, but we will only illustrate this in $\Pen_3$ as the other cases are proven similarly. The singular elements for this pencil are $\eta \in \{0,1,\infty\}$. Substituting,
    \begin{equation*}
        H_{2,3}(0,s) = s^4 , H_{2,3}(1,s) = (s-1)^4 , H_{2,3}(\infty,s)\neq 0 ,\, \forall s \in \F_q
    \end{equation*}
    where we see that $H_{2,3}(r,s) = 0$ if and only if $r=s$.
\end{proof}

We will use Lemma \ref{lem:h2_ss_0} to remove pairs $(r,s)$ that satisfy $H_{2,j}(r,s) = 0$ but should not be counted. Moreover, it proves that for the triangle case, all valid pairs that correspond to $3$-Poncelet pairs are smooth pairs. That is, there are no $3$-Poncelet pairs $(\A,\B)$ where $\A$ is singular and $\B$ is smooth. 

We can now start counting for each pencil $\Pen_j$ the number of valid pairs $(r,s)$ satisfying $H_{2,j}(r,s) = 0$. We will mainly use the discriminant $D(f)$ of a polynomial $f$ to facilitate the counting. In using discriminants, it will be important to know whether something is a non-zero square or zero in $\F_q$. With this, we define 

$Z^{*}_\varphi = \{s \in \F_q \, | \, \varphi(s) = y^2, \, y \in \F_q \setminus \{0\} \}$ = set of values which makes $\varphi(s)$ a non-zero square,

$Z_\varphi = \{s \in \F_q \, | \,  \varphi(s) = 0 \}$ = set of roots of $\varphi(s)$.

We first deal with the case where $H_{2,j}(r,s)$ is quadratic in $r$; that is, $j \in \{3,14,16,18,19\}$. Here we treat $H_{2,j}(r,s)$ as a polynomial in $(F_q[s])[r]$, the collection of polynomials in $r$ having polynomials in $s$ as coefficient. This makes the discriminant of $H_{2,j}(r,s)$ a function of $s$, and we denote this discriminant by $D_{j}(s)$. For a fixed $s$, $D_j(s)$ will tell us how many pairs $(r,s)$ satisfy $H_{2,j}(r,s) = 0$, and summing up all these counts for all valid values of $s$, we obtain the exact count that we want. In particular, each $s \in Z_{D_j}$ will contribute one pair and $s \in Z^{*}_{D_j}$ will contribute two pairs. Upon computing the discriminants, we observe that all of them take the form
\begin{equation}
    \label{eqn:discriminant_2}
    D_j(s) = 16 \pr{\kappa_j(s)}^2 \varphi_j(s),\, j \in \{3,14,16,18,19\}
\end{equation}
for some $\kappa_j(s), \varphi_j(s) \in \F_q[s]$.

We summarize in Tables \ref{tab:h2_root_kappa_summary} and \ref{tab:h2_disc_varphi_summary} useful information about $\kappa_j(s)$ and $\varphi_j(s)$, respectively.

\begin{table}[!htb]
\centering
\renewcommand{\arraystretch}{1.5}%
\caption{Summary of roots in $\F_q$ for $\kappa_j(s)$}
\begin{tabular}{c|l|l}
    $\Pen_j$ & $\kappa_{j}(s)^\dag$ & Roots in $\F_q$\\
    \hline
    $\Pen_{3}$ & $\kappa_{3}(s)=s(s-1)$ & $0,1$ \\
    \hline
    $\Pen_{14}$ & $\kappa_{14}(s)=e s^2 - s + 1$ & None \\
    \hline
    $\Pen_{16}$ & $\kappa_{16}(s)=s^2 - s + (-4 d e + d + e) $ & $\eta_{+} {}^{**},\eta_{-} {}^{**}$ \\
    \hline
    $\Pen_{18}$ & $\kappa_{18}(s)=s^3 + b s^2 + c s + 1  $ & None\\
    \hline
    $\Pen_{19}$ & $\kappa_{19}(s)=\nu s^2 + \rho s + \sigma^2$ & None\\
    \hline 
\end{tabular}

\label{tab:h2_root_kappa_summary}

\flushleft

\footnotesize{$^\dag$ $\kappa_j(s)$ as in Equation \ref{eqn:discriminant_2} }\\
\footnotesize{$^{**}$ $\eta_{+}, \eta_{-}$ as in Table \ref{tab:pencil_sing_base} }

\end{table}

\begin{table}[!htb]
\centering
\renewcommand{\arraystretch}{1.5}%
\caption{Summary of discriminants for $\varphi_j(s)$}
\begin{tabular}{c|l|l}
    $\Pen_j$ & $\varphi_{j}(s)^\dag$ & Discriminant $D(\varphi_j)$\\
    \hline
    $\Pen_{3}$ & $\varphi_{3}(s)=s^2 - s + 1$ & $-3$ \\
    \hline
    $\Pen_{14}$ & $\varphi_{14}(s)=e^2 s^2 - e s + (- 3 e + 1)$ & $-3e^2(1-4e)$ \\
    \hline
    $\Pen_{16}$ & $\varphi_{16}(s)=s^2 - s + (12 d e - 3 d - 3 e + 1) $ & $-3(1-4d)(1-4e)$ \\
    \hline
    $\Pen_{18}$ & $\varphi_{18}(s)=(b^2 - 3 c) s^2 + (b c - 9) s + (-3b + c^2) $ & $-3(b^2c^2-4c^3-4b^3-27+18bc)$\\
    \hline
    $\Pen_{19}$ & $\varphi_{19}(s)=\nu^2 s^2 + \nu \rho s + (- 3 \nu \sigma^2 + \rho^2)$ & $-3\nu^2(\rho^2 - 4\nu\sigma^2) $\\
    \hline 
\end{tabular}

\label{tab:h2_disc_varphi_summary}

\flushleft

\footnotesize{$^\dag$ $\varphi_j(s)$ as in Equation \ref{eqn:discriminant_2}}\\

\end{table}

We state important observations about $\kappa_j$ and $\varphi_j(s)$ in Lemma \ref{lem:kappa_0_sing} and Lemma \ref{lem:varphi_disc0_char3}.

\begin{lemma}
    \label{lem:kappa_0_sing}
    For any $\kappa_j(s)$ in Table \ref{tab:h2_root_kappa_summary} and $s \in \F_q$, then $\kappa_j(s) = 0$ if and only if $C_j(s)$, the matrix representation of the element corresponding to $s$ in pencil $\Pen_j$, is singular.
\end{lemma}

\begin{proof}
    Comparing the singular elements of the pencil in Table \ref{tab:pencil_sing_base} and the roots of $\kappa_j(s)$ in Table \ref{tab:h2_root_kappa_summary}, we can see that for $s \in \F_q, \, \kappa_j(s)$ is zero if and only if $C_j(s)$ is singular. We just need to show how we obtained the roots of $\kappa_j(s)$. 
    
    Roots of $\kappa_3$ are straightforward, and the quadratic formula is used to obtain those of $\kappa_{16}$. Referring to the assumptions on Table \ref{tab:pencil_class_auto}, $T^2 + T + e$ is irreducible, which implies that $1-4e$ is non-square in $\F_q$. But this is also the discriminant of $\kappa_{14}$, so it is also irreducible and does not have any root in $\F_q$. $\kappa_{18}$ follows from the assumption that $T^3 + bT^2 + cT + 1$ is irreducible. Finally, $\kappa_{19}$ has discriminant $\rho^2 - 4\nu\sigma^2$ which is assumed to be non-square.
\end{proof}

\begin{lemma}
    \label{lem:varphi_disc0_char3}
    Let $D(\varphi_j)$ be any discriminant in Table \ref{tab:h2_disc_varphi_summary}, then $D(\varphi_j) = 0$ if and only if char $\F_q = 3$.  That is, $\varphi_j(s)$ is a square of a linear polynomial in $\F_q[s]$ if and only if char $\F_q = 3$. 
\end{lemma}

\begin{proof}
    From the assumptions on Table \ref{tab:pencil_class_auto}, $T^2 + T + d$ and $T^2 + T + e$ are irreducible and we must have $1-4d, 1-4e$, and $e$ to be non-zero which are factors that appear on $D(\varphi_{14})$ and $D(\varphi_{16})$. By irreducibility of $T^3 + bT^2 + cT + 1$, a result from \cite{dic1906irredfinite} states that its discriminant, $b^2 c^2 - 4c^3 - 4b^3d - 27d^2 + 18 bc$, must be non-zero and this factor appears on $D(\varphi_{18})$. Finally, since $\nu$ and $\rho^2 - 4\nu\sigma^2$ are all non-square, then they are also non-zero, and these factors appear on $D(\varphi_{19})$.

    With this, we can see that all $D(\varphi_j)$ in Table \ref{tab:h2_disc_varphi_summary} have the form $-3$ multiplied by a non-zero factor. Hence, the whole quantity is 0 if and only if $-3$ is 0 in $\F_q$, which happens if and only if char $\F_q = 3$.
\end{proof}

A useful lemma adapted from the main argument in \cite{chi2017ptc} is given below as Lemma \ref{lem:varphi_square_count}.

\begin{lemma}
    \label{lem:varphi_square_count}
    Suppose $\varphi(s) = u_2 s^2 + u_1 s + u_0 \in \F_q[s]$ is a quadratic polynomial that is not a square of a linear polynomial in $\F_q[s]$ and $u_2$ is a non-zero square in $\F_q$.

    Let $Z^{*}_\varphi = \{s \in \F_q \, | \, \varphi(s) = y^2, \, y \in \F_q \setminus \{0\} \}$, and $Z_\varphi = \{s \in \F_q \, | \,  \varphi(s) = 0 \}$. Then, 

    \begin{equation*}
        \abs{Z^{*}_\varphi} = \frac{q-1-\abs{Z_\varphi}}{2}.
    \end{equation*}

\end{lemma}

\begin{proof}
    Recall that the non-zero square elements in $\F_q$ form a group under multiplication, and the product of a square and a non-square element of $\F_q$ will yield a non-square element. Since $u_2$ is a non-zero square, then $\varphi(s)$ is a non-zero square if and only if $\frac{\varphi(s)}{u_2}$ is a non-zero square. Hence, $\abs{Z^{*}_\varphi} = \abs{Z^{*}_{\frac{\varphi}{u_2}}}$. Moreover, the zeros of $\varphi(s)$ do not change when we divide it by a non-zero element of $\F_q$. So, we also have $\abs{Z_\varphi} = \abs{Z_{\frac{\varphi}{u_2}}}$. Thus, without loss of generality, we can assume that $u_2 = 1$; otherwise, we just consider $\frac{\varphi(s)}{u_2}$. 
    
    By completing the square, we obtain $\varphi(s) = (s-\beta)^2 +\gamma$ where $\beta, \gamma \in \F_q$ but $\gamma \neq 0$ by the assumption that $\varphi(s)$ is not a square of a linear polynomial. Now, we are interested in $s$ such that
    \begin{equation}
        \label{eqn:varphi_compsquare}
        \varphi(s) = (s-\beta)^2 +\gamma = y^2
    \end{equation}
    where $y \in \F_q$. Isolating $\gamma$ on one side of the equation gives us
    \begin{equation*}
        y^2 - (s-\beta)^2 = \underbrace{(y-s+\beta)}_{\omega}\underbrace{(y+s-\beta)}_{\frac{\gamma}{\omega}} = \gamma.
    \end{equation*}
    Since $\gamma$ is non-zero, the first factor $\omega = (y-s+\beta)$ is a non-zero element of $\F_q$ and the second factor $(y+s-\beta)$ can be expressed as $\dfrac{\gamma}{\omega}$. Thus, we get $y = \omega + s - \beta$ and substituting this back to Equation \ref{eqn:varphi_compsquare}, we obtain $(s-\beta)^2 +\gamma = (\omega + s - \beta)^2$ where we can solve for $s$ in terms of $\omega$ as
    \begin{equation}
        \label{eqn:varphi_square_param}
        s = \beta + \frac{\gamma}{2\omega} - \frac{\omega}{2}.
    \end{equation}
    Here, each $\omega \in \F_q \setminus \{0\}$ corresponds to a value of $s$ that will make $\varphi(s)$ a square in $\F_q$. In this case, for a fixed $s$ where $\varphi(s)$ is square, Equation \ref{eqn:varphi_square_param} can be treated as quadratic polynomial in $\omega$ with the form
    \begin{equation*}
        \omega^2 + (2s-2\beta)\omega -\gamma = 0
    \end{equation*}
    and discriminant $4\varphi(s)$. Hence, each $s$ making $\varphi(s)$ a non-zero square will have two corresponding $\omega$ but only one $\omega$ corresponding to each $s$ for which $\varphi(s) = 0$. With this, we have $q-1$ possible values for $\omega$ corresponding to the elements of $\F_q \setminus \{0\}$ and we remove one $\omega$ for each $s$ corresponding to a root of $\varphi(s)$ which leaves us the remaining $q-1-\abs{Z_\varphi}$ values of $\omega$ corresponding to $s$ which makes $\varphi(s)$ a non-zero square. Since we only get one $s$ per two values of $\omega$ in this case, we obtain the final count to be
    \begin{equation*}
        \abs{Z^{*}_\varphi} = \frac{q-1-\abs{Z_\varphi}}{2}.
    \end{equation*}    
    
\end{proof}

Now we are ready to count valid pairs $(r,s)$ that correspond to $3$-Poncelet pairs. For this purpose, we denote the set of valid pairs that correspond to $3$-Poncelet pairs as $\Gamma_j(H_2) = \{(r,s)\in \Phi_j \, | \, H_{2,j}(r,s) = 0 \}$, which we can decompose into three disjoint sets

$\Gamma_j(H_2(\cdot,\cdot)) = \{(r,s)\in \Gamma_j(H_2) | \, r,s \in \F_q \}$,

$\Gamma_j(H_2(\cdot,\infty)) = \{(r,s)\in \Gamma_j(H_2) | \, r \in \F_q , \,  s = \infty \}$, and

$\Gamma_j(H_2(\infty,\cdot)) = \{(r,s)\in \Gamma_j(H_2) | \, r = \infty, \, s \in \F_q \}$.

\begin{proposition}
    \label{prop:n3_char_n3_case_3_14_16_19}
    For char $\F_q \neq 3$ and $j \in \{3,14,16,19\}$, $\abs{\Gamma_j(H_2)} =  q-1-2\abs{Z_{\kappa_j}}$ with $\kappa_j$ as shown in Table \ref{tab:h2_root_kappa_summary}. In particular,
    \begin{equation*}
        \abs{\Gamma_j(H_2)} = q-5, \, j \in \{3,16\}, \text{ and}
    \end{equation*}
    \begin{equation*}
        \abs{\Gamma_j(H_2)} = q-1, \, j \in \{14,19\}.
    \end{equation*}
\end{proposition}

\begin{proof}
    By Lemma \ref{lem:kappa_0_sing}, $\kappa_j(s) \neq 0$ for all valid pairs $(r,s)$. Moreover, we can see that all $s \in \F_q$ corresponding to a singular conic is an element of $Z^{*}_{\varphi_j}$ since
    \begin{equation*}
        \varphi_{3}(0)=\varphi_{3}(1) = 1, \, \varphi_{16}(\eta_+) = \varphi_{16}(\eta_-) = (1-4d)(1-4e)
    \end{equation*}
    which by the irreducibility condition in Table \ref{tab:pencil_class_auto}, implies that both $(1-4e)$ and $(1-4d)$ are non-squares and so their product is a non-zero square. Thus, we remove the elements of $Z_{\kappa_j}$ from $Z^{*}_{\varphi_j}$ since they do not correspond to a valid pair.
    
    Since the form of our discriminant $D_j(s)$ is a product of $\varphi_j(s)$ and $16\pr{\kappa_j(s)}^2$, where we have shown that the latter factor is always a non-zero square, we have that each element of $Z_{\varphi_j}$ will contribute one pair $(r,s)$ since it will make $D_j(s) = 0$. Meanwhile, elements of $Z^{*}_{\varphi_j}$ that are not in $Z_{\kappa_j}$ will contribute two pairs of $(r,s)$ since this will make $D_j(s)$ a non-zero square. Since $H_2(\infty,s) \neq 0$ for all $s \neq \infty$ and $\infty$ corresponds to a singular conic for the pencils in our current case, then $\infty$ does not contribute to any pairs. Thus,
    \begin{align*}
        \abs{\Gamma_j(H_2)} &= \abs{\Gamma_j(H_2(\cdot,\cdot))} + \underbrace{\abs{\Gamma_j(H_2(\cdot,\infty))}}_{=0} + \underbrace{\abs{\Gamma_j(H_2(\infty,\cdot))}}_{=0} \\
        &= 2\pr{\abs{Z^{*}_{\varphi_j}} - \abs{Z_{\kappa_j}}} + \abs{Z_{\varphi_j}}.
    \end{align*}
    
    By Lemma \ref{lem:varphi_disc0_char3}, all $\varphi_j(s)$ are not a square of a linear polynomial, and since each $\varphi_j(s)$ will have a coefficient of $s^2$ that is a non-zero square in $\F_q$, we can use Lemma \ref{lem:varphi_square_count}. All pairs $(r,s)$ in $\Gamma_j(H_2)$ satisfy $r \neq s$ since char $\F_q \neq 3$ and by Lemma \ref{lem:h2_ss_0}, $H_2(s,s)$ only happens when $s$ corresponds to a non-singular conic which we already removed by removing elements in $Z_{\kappa_j}$. Therefore, we obtain the final count as
    \begin{equation*}
        \abs{\Gamma_j(H_2)} = 2\pr{\underbrace{\frac{q-1-\abs{Z_{\varphi_j}}}{2}}_{\text{Lemma } \ref{lem:varphi_square_count}} - \abs{Z_{\kappa_j}}} + \abs{Z_{\varphi_j}}  = q-1-2\abs{Z_{\kappa_j}}.
    \end{equation*}
\end{proof}

\begin{proposition}
    \label{prop:n3_char_e3_case_3_14_16_19}
    For char $\F_q = 3$ and $j \in \{3,14,16,19\}$, $\abs{\Gamma_j(H_2)} = q - \abs{Z_{D_j}}$ with $D_j$ as defined in Equation \ref{eqn:discriminant_2}. In particular,
    \begin{equation*}
        \abs{\Gamma_j(H_2)} = q-3, \, j \in \{3,16\}, \text{ and}
    \end{equation*}
    \begin{equation*}
        \abs{\Gamma_j(H_2)} = q-1, \, j \in \{14,19\}.
    \end{equation*}
\end{proposition}

\begin{proof}
    By Lemma \ref{lem:varphi_disc0_char3}, all $D_j(s)$ will always be a square. In effect, $s \in \F_q$ will contribute two pairs if it is not a root of $D_j(s)$, and each root contributes one pair. $D_3(s)$ has three roots in $\F_q$ which are $0,1,-1$. $D_{16}(s)$ also has three roots in $\F_q$ given by $ \eta_{+},\eta_{-}, -1 $ with $\eta_{+}, \eta_{-}$ defined as in Table \ref{tab:pencil_sing_base}. $D_{14}(s)$ and $D_{19}(s)$ both only have one root in $\F_q$ given by $-e^{-1}$ and $\nu^{-1}\rho$, respectively.

    Since char $\F_q = 3$, Lemma \ref{lem:h2_ss_0}, states that all $s\in \F_q$ satisfies $H_{2,j}(s,s) = 0$ and hence, we should remove $q$ pairs in counting since these are not valid pairs. Lastly, $\infty$ does not contribute into any pairs by the same arguments in the proof of Proposition \ref{prop:n3_char_n3_case_3_14_16_19}. Thus,
    \begin{align*}
        \abs{\Gamma_j(H_2)} &= \abs{\Gamma_j(H_2(\cdot,\cdot))} + \underbrace{\abs{\Gamma_j(H_2(\cdot,\infty))}}_{=0} + \underbrace{\abs{\Gamma_j(H_2(\infty,\cdot))}}_{=0} \\
        &= 2\pr{q - \abs{Z_{D_j}}} + \abs{Z_{D_j}} -\underbrace{q}_{\text{Lemma } \ref{lem:h2_ss_0}} = q - \abs{Z_{D_j}}.
    \end{align*}    
\end{proof}

Before dealing with $\Pen_{18}$, we prove this lemma to simplify our proof.

\begin{lemma}
    \label{lem:choice_bc}
    Suppose char $\F_q \neq 2$. Then we can always find $b \in \F_q \setminus \{0\}$ such that $T^3 + bT^2 + 1$ is an irreducible polynomial in $\F_q[T]$. 
\end{lemma}

\begin{proof}
    Note that $T=0$ is not a root of $T^3 + bT^2 + 1$, so the only possible roots are the non-zero elements of $\F_q$, which is generated by the primitive element $\alpha$ and each of them takes the form $\alpha^k$ for $k \in \{1,...,q-1\}$. Now, a cubic polynomial in $\F_q[T]$ is irreducible if and only if it does not have any root in $\F_q$. But for $T = \alpha^k$, for some $k \in \{1,...,q-1\}$,
    \begin{equation*}
        \alpha^{3k} + b \alpha^{2k} + 1 = 0 \iff b = -\alpha^k - \alpha^{-2k}.
    \end{equation*}
    Hence, if we choose $b \in \F_q \setminus \{-\alpha^k - \alpha^{-2k}\,|\,k \in \{1,...,q-1\}\}$, then this $b$ satisfies the condition that $T^3 + bT^2 + 1$ does not have a root in $\F_q$, making it irreducible in $\F_q[T]$. This choice is possible since $\{-\alpha^k - \alpha^{-2k}\,|\,k \in \{1,...,q-1\}\}$ has at most $q-1$ elements while $\F_q$ has $q$ elements. Lastly, $b$ cannot be zero since $T^3 + 1$ is not irreducible.
\end{proof}

\begin{proposition}
    \label{prop:n3_char_n3_case_18}
    For char $\F_q \neq 3$, $\abs{\Gamma_{18}(H_2)} = q + 1$.
\end{proposition}

\begin{proof}
    By Lemma \ref{lem:choice_bc}, we can choose $c=0$ and $b \neq 0$ such that $T^3 + bT^2 + 1$ is irreducible. Referring to Table \ref{tab:h2_summary}, observe that $H_2(r,s) = H_{2,18}(\infty,s) r^2 + f_1(s) r + f_0(s)$, where $f_1(s),f_0(s) \in \F_q[s]$. 
    
    First, let $s \in \F_q$ such that $H_{2,18}(\infty,s) \neq 0$. In this case, $H_{2,18}(r,s)$ is a quadratic polynomial in $r$ so we can interpret the discriminant $D_{18}(s)$ as in Equation \ref{eqn:discriminant_2}. By Lemma \ref{lem:kappa_0_sing}, $\kappa_{18}(s)$ is always non-zero and thus, $D_{18}(s)$ is a non-zero square or zero if and only if $\varphi_{18}(s)$ is also a non-zero square or zero. Moreover, our choice of $b$ and $c$ allow us to use Lemma \ref{lem:varphi_square_count} since the coefficient of $s^2$ in $\varphi_{18}(s)$ will just be $b^2$.

    For $s \in \F_q$ such that $H_{2,18}(\infty,s) = 0$, \cite{wan2025pnc} pointed out that we have $D_{18}(s) = f_1(s)^2$ is always a square and this implies that for such $s$, $\varphi_{18}(s)$ is a square. We will show that $s$ is not a root of $\varphi_{18}(s)$ and hence, must be in $Z^{*}_{\varphi_{18}}$. We do this by computing the resultant and recalling that $\resultant(q_1(s),q_2(s),s) = 0$ if and only if $q_1(s)$ and $q_2(s)$ share the same root as a polynomial in $s$ \cite{ken2011algebracurve}. Now,
    \begin{equation*}
        \resultant(H_{2,18}(\infty,s),\varphi_{18}(s),s) = b^4(4b^3 + 27)^2
    \end{equation*}
    but $b\neq0$ by our choice and $4b^3+27 \neq 0$ by \cite{dic1906irredfinite} since it is the discriminant of the irreducible polynomial $T^3 + bT^2 + 1$. Hence, the two polynomials do not share any root, and we have the case where $s \in Z^{*}_{\varphi_{18}}$ but should not be counted since such $s$ will not make $H_{2,18}(r,s)$ quadratic in $r$.
    
    Now, we will show that for $s \in \F_q$ such that $H_{2,18}(\infty,s) = 0$, $H_2(r,s)$ will be a linear function of $r$. To do this, we use the resultant again and see that
    \begin{equation*}
        \resultant(H_{2,18}(\infty,s),f_1(s),s) = 16b^2(4b^3+27)^3
    \end{equation*}
    which is also non-zero by the same argument as the previous resultant. This implies that $H_2(\infty,s)$ and $f_1(s)$ cannot be simultaneously zero. Hence, for each element $(\infty,s) \in \Gamma_{18}(H_2(\infty,\cdot))$, we obtain an additional pair of solution $(r,s) \in \Gamma_{18}(\cdot,\cdot)$ corresponding to the solution of the linear equation which yields $r = \frac{f_0(s)}{f_1(s)}$.
    
    Finally, we have $H_{2,18}(r,\infty)$ which is a quadratic polynomial with discriminant that reduces to $(4b)^2$, a non-zero square by our choice of $b$ and $c$, giving us an additional 2 pairs. The pairs $(r,s)$ we obtain here will have $r \neq s$ by Lemma \ref{lem:h2_ss_0} since we have no singular conics for this pencil and char $\F_q \neq 3$. Taking all of this into consideration, we obtain the count
    \begin{align*}
       \abs{\Gamma_{18}(H_2)} &= \abs{\Gamma_{18}(H_2(\cdot,\cdot))} + \underbrace{\abs{\Gamma_{18}(H_2(\cdot,\infty))}}_{=2} + \abs{\Gamma_{18}(H_2(\infty,\cdot))}\\
        &= \underbrace{2\pr{\abs{Z^{*}_{\varphi_{18}}} - \abs{\Gamma_{18}(H_2(\infty,\cdot))}} + \abs{Z_{\varphi_{18}}}}_{H_{2,18}(r,s) \text{ quadratic in } r} + \underbrace{\abs{\Gamma_{18}(H_2(\infty,\cdot))}}_{H_{2,18}(r,s) \text{ linear in } r} + \, 2 \, + \abs{\Gamma_{18}(H_2(\infty,\cdot))}\\
        &= 2\pr{\underbrace{\frac{q-1-\abs{Z_{\varphi_{18}}}}{2}}_{\text{Lemma } \ref{lem:varphi_square_count}} - \abs{\Gamma_{18}(H_2(\infty,\cdot))}} + \abs{Z_{\varphi_{18}}} + 2\abs{\Gamma_{18}(H_2(\infty,\cdot))} + 2 = q+1.
    \end{align*}
\end{proof}

\begin{proposition}
    \label{prop:n3_char_e3_case_18}
    For char $\F_q = 3$, $\abs{\Gamma_{18}(H_2)} = q$.
\end{proposition}

\begin{proof}
    By Lemma \ref{lem:choice_bc}, we can choose $c=0$ and $b \neq 0$ such that $T^3 + bT^2 + 1$ is irreducible. Observe that $H_{2,18}(\infty,0) = 4b \neq 0$. Thus, $0 \notin \Gamma_{18}(H_2(\infty,\cdot))$. For $s \in \F_q \setminus \Gamma_{18}(H_2(\infty,\cdot))$, $H_{2,18}(r,s)$ is quadratic in $r$ and $D_{18}(s) = (4(s^3+bs^2+1)(bs))^2$ will always be a square and is zero exactly when $s = 0$, which we have shown to be not in $\Gamma_{18}(H_2(\infty,\cdot))$. Hence, each non-zero $s$ not in $\Gamma_{18}(H_2(\infty,\cdot))$ will contribute two pairs $(r,s)$ and we get one pair for $s = 0$.

    Now, $\resultant(H_{2,18}(\infty,s),f_1(s),s) = 64b^{11}$ in char $\F_q = 3$ which is non-zero. Thus, we get one pair $(r,s)$ where $r = \frac{f_0(s)}{f_1(s)}$ is the solution to the linear equation $H_{2,j}(r,s) = 0$, for each $s \in \Gamma_{18}(H_2(\infty,\cdot))$.
    
    Considering the quadratic polynomial $H_2(r,\infty)$, the discriminant is still $(4b)^2$, which is a non-zero square. This gives us an additional two pairs.
    
    Finally, by Lemma \ref{lem:h2_ss_0}, we need to remove the pairs $(s,s)$ where $s \in \Prj^1(\F_q)$ which removes a total of $q+1$ pairs. Taking all these into account, we obtain
    
    \begin{align*}
        \abs{\Gamma_j(H_2)} &= \abs{\Gamma_j(H_2(\cdot,\cdot))} + \underbrace{\abs{\Gamma_j(H_2(\cdot,\infty))}}_{=2} + \abs{\Gamma_j(H_2(\infty,\cdot))} \\
        &= \underbrace{2\pr{q-1 - \abs{\Gamma_j(H_2(\infty,\cdot))}} + 1}_{H_{2,18}(r,s) \text{ quadratic in } r} + \underbrace{\abs{\Gamma_j(H_2(\infty,\cdot))}}_{H_{2,18}(r,s) \text{ linear in } r} \\
        & \hspace{10pt} +\, 2 + \abs{\Gamma_j(H_2(\infty,\cdot))} - \underbrace{\pr{q+1}}_{\text{Lemma } \ref{lem:h2_ss_0}}  = q.
    \end{align*}
\end{proof}

\begin{proposition}
    \label{prop:n3_char_n3_case_4_6_15_17}
    For char $\F_q \neq 3$ and $j \in \{4,6,15,17\}$, $\abs{\Gamma_j(H_2)} = q -1$.
\end{proposition}

\begin{proof}
    All of these share the same form of $H_{2,j}(r,s)$, which is degree 1 in both $r$ and $s$ with the roles of $r$ and $s$ reversed for $j=17$. Choosing $s \in \F_q \setminus \{0\}$, $H_{2,j}(r,s) = 0$ if and only if $r = \frac{s}{4}$ for $j \neq 17$ and $r = 4s$ for $j = 17$. Thus, there is only one value of $r$ for each $s$, giving us a total of $q-1$ pairs. We do not count the pairs of the form $(\infty,s)$ for $j \neq 17$ and $(0,s)$ for $j=17$ since they correspond to the irregular pairs.
\end{proof}

\begin{proposition}
    \label{prop:n3_char_e3_case_4_6_15_17}
    For char $\F_q = 3$ and $j \in \{4,6,15,17\}$, $\abs{\Gamma_j(H_2)} = 0$.
\end{proposition}

\begin{proof}
    If char $\F_q = 3$, we get the condition that $H_{3,j}=0$ if and only if $r = s$, which cannot happen since valid pairs need to be distinct. We can also use Proposition \ref{prop:n3_char_n3_case_4_6_15_17}, which gives us $q-1$ pairs, but since char $\F_q = 3$, Lemma \ref{lem:h2_ss_0} states that $q-1$ pairs should be removed.
\end{proof}

\begin{proposition}
    \label{prop:n3_char_n3_case_5_8}
    For char $\F_q \neq 3$ and $j \in \{5,8\}$, $\abs{\Gamma_j(H_2)} = 0$.
\end{proposition}

\begin{proof}
    All valid pairs $(r,s)$ will have $H_{2,j}(r,s) = 3 \neq 0$. Pairs of the form $(\infty,s)$ correspond to irregular pairs and will not be counted.
\end{proof}

\begin{proposition}
    \label{prop:n3_char_e3_case_5_8}
    For char $\F_q = 3$ and $j \in \{5,8\}$, $\abs{\Gamma_j(H_2)} = q(q-1)$.
\end{proposition}

\begin{proof}
    All valid pairs $(r,s)$ will satisfy $H_{2,j}(r,s) = 0$. Pairs of the form $(\infty,s)$ correspond to irregular pairs, which we do not count. Since the only singular conic in this case corresponds to $\infty$, then the only condition we need to satisfy is that $r,s \in \F_q$ with $r \neq s$, giving us a total of $q(q-1)$ pairs.
\end{proof}

We summarize all the counts derived for the triangle case in Theorem \ref{thm:n3_counts}.

\begin{theorem}
\label{thm:n3_counts}
Any valid pair corresponding to a $3$-Poncelet pair must be a smooth pair. The number of pairs corresponding to $3$-Poncelet pairs among the valid pairs, for each pencil, is summarized in Table \ref{tab:n3_count}.

\begin{table}[!htb]
\centering
\renewcommand{\arraystretch}{2}%
\caption{Count of $3$-Poncelet pairs for each pencil in $\Prj^2(\F_q)$}
\begin{tabular}{l|c|c}
\hline
$\Pen_j$ & char $\F_q \neq 3$ & char $\F_q = 3$ \\
\hline
$\Pen_{3}, \Pen_{16}$       &  $q-5$  & $q-3$  \\
\hline
$\Pen_{4}, \Pen_{6}, \Pen_{15}, \Pen_{17}$      &   $q-1$           &  $0$           \\
\hline
$\Pen_{5}, \Pen_{8}$      &  $0$            &    $q(q-1)$         \\
\hline
$\Pen_{14}, \Pen_{19}$       &  $q-1$            &  $q-1$           \\
\hline
$\Pen_{18}$       &   $q+1$           &    $q$  \\
\hline
\end{tabular}

\label{tab:n3_count}
\end{table}

\end{theorem}

\begin{proof}
    The fact that all valid pairs corresponding to $3$-Poncelet pairs are smooth pairs follows from Lemma \ref{lem:h2_ss_0}. The counts are obtained from Propositions \ref{prop:n3_char_n3_case_3_14_16_19}, \ref{prop:n3_char_e3_case_3_14_16_19}, \ref{prop:n3_char_n3_case_18}, \ref{prop:n3_char_e3_case_18}, \ref{prop:n3_char_n3_case_4_6_15_17}, \ref{prop:n3_char_e3_case_4_6_15_17}, \ref{prop:n3_char_n3_case_5_8}, and \ref{prop:n3_char_e3_case_5_8}.
\end{proof}

Using Theorem \ref{thm:n3_counts}, we can now compute the probabilities of obtaining a $3$-Poncelet pair in pencil $\Pen_j$ under the probability space stated at the beginning of Section \ref{sec:prob}. In particular, considering the sample space of valid pairs $\Phi_j$ and the uniform measure, the probability of obtaining a $3$-Poncelet pair in pencil $\Pen_j$ is given by
\begin{equation*}
    \Prb(\Gamma_{j}(H_2)|\Phi_{j}) = \dfrac{\abs{\Gamma_{j}(H_2)}}{\abs{\Phi_{j}}}.
\end{equation*}

Upon restricting the sample space to smooth pairs $\Psi_j$, the probability of obtaining a $3$-Poncelet pair in pencil $\Pen_j$ is given by
\begin{equation*}
    \Prb(\Gamma_{j}(H_2)|\Psi_{j}) = \dfrac{\abs{\Gamma_{j}(H_2) \cap \Psi_{j}}}{\abs{\Psi_{j}}}
\end{equation*}
with $\abs{\Gamma_{j}(H_2) \cap \Psi_{j}} = \abs{\Gamma_{j}(H_2)}$ by Theorem \ref{thm:n3_counts}. 

We summarize these probabilities in Corollary \ref{cor:n3_prob}. Note that the counts used as numerators in the probability are the same for both the valid pairs and smooth pairs, but the denominators differ, and we need to adjust the counting for the irregular pairs.

\begin{corollary}
    \label{cor:n3_prob}
    The probabilities of obtaining a valid pair or a smooth pair corresponding to a $3$\nobreakdash-Poncelet pair for each pencil $\Pen_j$ under our assumed probability space are summarized in Table \ref{tab:n3_prob}.

    \begin{table}[!htb]
        \centering
        \renewcommand{\arraystretch}{2}%
        \caption{Probabilities of obtaining a $3$-Poncelet pair for each pencil in $\Prj^2(\F_q)$}
        \begin{tabular}{l|cc|cc}
        \hline
        \multicolumn{1}{l|}{\multirow{2}{*}{$\Pen_j$}} & \multicolumn{2}{c|}{Smooth Pairs} & \multicolumn{2}{c}{Valid Pairs}  \\ \cline{2-5} 
        \multicolumn{1}{c|}{}                        & \multicolumn{1}{c|}{$ \chr\F_q \neq 3$}  & $ \chr\F_q = 3$& \multicolumn{1}{c|}{$ \chr\F_q \neq 3$} & $ \chr\F_q = 3$ \\ \hline
            $\Pen_{3}, \Pen_{16}$          & \multicolumn{1}{c|}{$\frac{q-5}{(q-2)(q-3)}$} &  $\frac{1}{q-2}$ for $q > 3 ^{\dag}$  & \multicolumn{1}{c|}{$\frac{q-5}{q(q-2)}$}   &   $\frac{q-3}{q(q-2)}$   \\ \hline
            $\Pen_{4}, \Pen_{15}$     & \multicolumn{1}{c|}{$\frac{1}{q-2}$} &  $0$  & \multicolumn{1}{c|}{$\frac{1}{q-1}$}   &   $0$   \\ \hline
            $\Pen_{5}, \Pen_{8}$      & \multicolumn{1}{c|}{$0$} &  $1$  & \multicolumn{1}{c|}{$0$}   &   $1$   \\ \hline
            $\Pen_{6}, \Pen_{17}$     & \multicolumn{1}{c|}{$\frac{1}{q-2}$} &  $0$  & \multicolumn{1}{c|}{$\frac{1}{q-2}$}   &   $0$   \\ \hline
            $\Pen_{14}, \Pen_{19}$       & \multicolumn{1}{c|}{$\frac{1}{q}$} &  $\frac{1}{q}$  & \multicolumn{1}{c|}{$\frac{q-1}{q^2}$}   &   $\frac{q-1}{q^2}$   \\ \hline
            $\Pen_{18}$        & \multicolumn{1}{c|}{$\frac{1}{q}$} &  $\frac{1}{q+1}$  & \multicolumn{1}{c|}{$\frac{1}{q}$}   &   $\frac{1}{q+1}$   \\ \hline
        \end{tabular}
        
        \label{tab:n3_prob}
        \flushleft
        \footnotesize{${}^{\dag}$: The sample space is empty for $q=3$}
    \end{table}
    
\end{corollary}

\begin{remark}
    Following the approach of \cite{chi2017ptc} of combining probabilities from multiple pencils, we recover their main result that the asymptotic probability of obtaining a $3$-Poncelet pair among smooth, transversally intersecting conics in $\Prj^2(\F_q)$ is $\frac{1}{q}$ when char~$\F_q \neq 3$. Under the same approach, our computation suggests that this asymptotic probability remains to be $\frac{1}{q}$ even if char~$\F_q = 3$.
\end{remark}

\subsection{Probabilities for Poncelet Tetragon}

We derive the form of $H_3$ from Cayley's condition in terms of $r$ and $s$ for each pencil $\Pen_j$. We use the notation $H_{3,j}$ to refer to the form of $H_3$ for pencil $\Pen_j$ and summarize them in Table \ref{tab:h3_summary}.

\begin{table}[!htb]
\centering
\renewcommand{\arraystretch}{1.5}%
\caption{Summary of Cayley's condition for the tetragon case}
\begin{tabular}{c|l}
    \hline
    $\Pen_j$ & $H_{3,j}{}^{\dag}$ \\
    \hline
    $\Pen_{3}$ & $H_{3,3} (r,s) = [r - s^2] [ (2 s - 1)r  - s^2] [r + (s^2 - 2 s)]$ \\
    & $H_{3,3}(\infty,s) = 2 s - 1$ \\
    \hline
    $\Pen_{4}$ & $ H_{3,4} (r,s) = 2 r - s$ \\
    & $H_{3,4}(\infty,s) = 0$ \\
    \hline
    $\Pen_{5}$ & $ H_{3,5} (r,s) = u {}^{\ddag}$ \\
    & $H_{3,5}(\infty,s) = 0$ \\
    \hline
    $\Pen_{6}$ & $ H_{3,6} (r,s) = 2 r - s$ \\
    & $H_{3,6}(\infty,s) = 0$ \\
    \hline
    $\Pen_{8}$ & $ H_{3,8} (r,s) = u {}^{\ddag}$ \\
    & $H_{3,8}(\infty,s) = 0$ \\
    \hline 
    $\Pen_{14}$ & $ H_{3,14} (r,s) = \, [ (2 e s - 1) r + (-e s^2 + 1)] [ (1 - 4 e) r^2 + (8 e s - 2 s) r + (-e^2 s^4 + 2 e s^3 - 6 e s^2 + 2 s - 1)]$ \\
    & $H_{3,14}(\infty,s) = 2 e s - 1$ \\
    \hline
    $\Pen_{15}$ & $ H_{3,15} (r,s) = 2 r - s$ \\
    & $H_{3,15} (\infty,s) = 0$ \\
    \hline
    $\Pen_{16}$ & $ \begin{aligned}
        H_{3,16} (r,s) = \, & [(2 s - 1)r + (- s^2 -4 d e + d + e)]  \\
        & [ (1-4d)(1-4e) r^2 + ((-32 d e  + 8 d  + 8 e  - 2 )s) r + (-s^4 + 2 s^3 + (24de - 6d - 6e)s^2 \\
        & + (-8de + 2d + 2e)s + (-16d^2e^2 + 8d^2e - d^2 + 8de^2 - 2de - e^2)] \\
    \end{aligned}$ \\
    & $H_{3,16} (\infty,s) = 2 s - 1$ \\
    \hline
    $\Pen_{17}$ & $ H_{3,17}(r,s) = r^2 (r - 2 s)$ \\
    & $H_{3,17}(\infty,s) = u {}^{\ddag}$ \\
    \hline 
    $\Pen_{18}$ & $\begin{aligned}
        H_{3,18} (r,s) = \, & [s^6 + 2bs^5 +5c s^4 + 20 s^3 + (20b-5c^2)s^2 + (8b^2 - 2bc^2 - 4c)s + (4bc - c^3 - 8)] r^3  \\
        &+ [b s^6 + (4b^2 - 6c)s^5 + (10bc-45)s^4 +(-40b + 20c^2)s^3 \\
        &+ (-20b^2 + 5bc^2 + 30c)s^2 + (-8bc + 2c^3 + 36)s +(4b - c^2)] r^2\\
        &+ [(-b^2+4c)s^6 + (2b^3 - 8bc + 36)s^5 + (5b^2c + 30b - 20c^2)s^4 \\
        &+ (20b^2 - 40c)s^3 + (10bc - 45)s^2 + (-6b+4c^2)s + c] r\\
        &+ [(-b^3 + 4bc - 8)s^6 + (-2b^2c - 4b + 8c^2)s^5 \\
        &+ (-5b^2 + 20c)s^4 + 20 s^3 + 5b s^2 + 2c s + 1]
    \end{aligned}$ \\
    & $H_{3,18} (r,\infty) = r^3 + br^2 + (-b^2 + 4c)r + (-b^3 + 4bc - 8)$ \\
    & $\begin{aligned}
        H_{3,18} (\infty,s) = \, & s^6 + 2b s^5 +5c s^4 + 20 s^3 + (20b-5c^2)s^2 + (8b^2 - 2bc^2 - 4c)s + (4bc - c^3 - 8)
    \end{aligned}$ \\
    \hline
    $\Pen_{19}$ & $  \begin{aligned}
        H_{3,19} (r,s) = \, & [ (2 \nu s + \rho)r + (-\nu s^2 + \sigma^2)]  \\
        & [(\rho^2 - 4 \nu \sigma^2) r^2  + ((8 \nu \sigma^2 - 2 \rho^2)s)r + (-\nu^2 s^4 - 2 \nu \rho s^3 -6 \nu \sigma^2 s^2 -2\rho\sigma^2 - \sigma^4)]
    \end{aligned}$ \\
    & $H_{3,19}(\infty,s) = 2 \nu s + \rho$ \\
    \hline 
\end{tabular}

\label{tab:h3_summary}

\flushleft

\footnotesize{$^\dag$ $H_{3,j} = H_3$ from Cayley's condition for pencil $\Pen_j$ where non-zero scalar multiples and denominators are ignored}\\
\footnotesize{$^\ddag$ $u$ is a non-zero constant for char $\F_q \neq 2$} \\

\end{table}

\begin{example}
    As an example of a valid $4$-Poncelet pair that is not a smooth pair, consider $q = 11$ and the pair $(\infty,6)$ in pencil $\Pen_3$. From Table \ref{tab:h3_summary}, we can see that $H_{3,3}(\infty,6) = 0$ implying that the conic pair $(\infty,6)$ in pencil $\Pen_3$ correspond to a $4$-Poncelet pair. We illustrate a Poncelet tetragon that can be constructed from this $4$-Poncelet pair in Figure \ref{fig:poncelet_n4_p3}. 
\end{example}

\begin{remark}
    In Figure \ref{fig:poncelet_n4_p3}, we see that the vertices of our Poncelet tetragon alternate between the two lines making up our singular conic $\A$. This behavior persists for $n > 4$, and in general, we can only inscribe a Poncelet $n$-gon in a singular conic $\A$ if $n$ is even. This has been proven in \cite{drarad2025poncelet} for the real projective plane.
\end{remark}

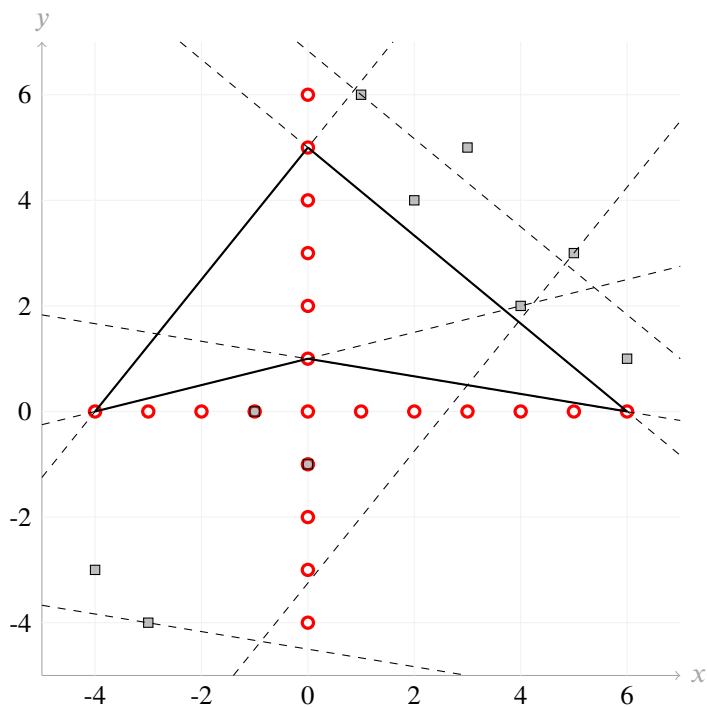
\begin{figure}[!htb]
\centering
\begin{tikzpicture}[scale=0.7]

\def\xmin{-5} \def\xmax{7}
\def\ymin{-5} \def\ymax{7}

  \draw[->, gray!70] (\xmin,-5) -- (\xmax,-5) node[right] {$x$};
  \draw[->, gray!70] (-5,\ymin) -- (-5,\ymax) node[above] {$y$};

\def\labelsep{1} 

\foreach \x in {-4,-2,0,2,4,6} {
    \draw[gray!70] (\x,-4);
    \node[below] at (\x,-4-\labelsep) {\small \x}; 
}

\foreach \y in {-4,-2,0,2,4,6} {
    \draw[gray!70] (-4,\y);
    \node[left] at (-4-\labelsep,\y) {\small \y}; 
}

  \foreach \x in {-4,-2,0,2,4,6} \draw[gray!10] (\x,\ymin) -- (\x,\ymax);
  \foreach \y in {-4,-2,0,2,4,6} \draw[gray!10] (\xmin,\y) -- (\xmax,\y);

  \foreach \x/\y in {0/0, 0/1, 0/2, 0/3, 0/4, 0/5, 0/6, 0/-4, 0/-3, 0/-2, 0/-1,
                     1/0, 2/0, 3/0, 4/0, 5/0, 6/0, -4/0, -3/0, -2/0, -1/0} {
    \filldraw[white, fill opacity=0.3, draw=red, very thick] (\x,\y) circle (3pt);
  }

  \foreach \x/\y in {0/-1, 1/6, 2/4, 3/5, 4/2, 5/3, 6/1, -4/-3, -3/-4, -1/0} {
    \filldraw[gray, fill opacity=0.5, draw=black] (\x,\y) ++(-2.5pt,-2.5pt)  rectangle ++(5pt,5pt);
  }

  \draw[dashed, thin] (-5,-1.25) -- (1.6,7);
  \draw[dashed, thin] (-1.4,-5) -- (7,5.5);
  
  \draw[dashed, thin] (-5,-0.25) -- (7,2.75);
  
  \draw[dashed, thin] (-5,1.83) -- (7,-0.17);
  \draw[dashed, thin] (-5,-3.67) -- (3,-5);
  
  \draw[dashed, thin] (-2.4,7) -- (7,-0.83);
  \draw[dashed, thin] (-0.2,7) -- (7,1);

  \draw[thick] (0,5) -- (-4,0);
  \draw[thick] (-4,0) -- (0,1);
  \draw[thick] (0,1) -- (6,0);
  \draw[thick] (6,0) -- (0,5);

\end{tikzpicture}

\caption{Poncelet tetragon in $\mathbb{F}_{11}^2$}
\label{fig:poncelet_n4_p3}

\flushleft
\vspace{0.1cm}
\footnotesize{$\bullet$ Conics are viewed in the affine plane $z=1$. $[x:y:1]$ is associated with the ordered pair $(x,y) \in \mathbb{F}_{11}^2$.}\\
\footnotesize{$\bullet$ Circles are the points of conic $\A: xy=0 $ and filled squares are the points of conic $\B: 6xy+xz+yz+z^2=0 $ in $ \mathbb{F}_{11}^2$.}\\
\footnotesize{$\bullet$ Broken lines represent the tangent lines to $\mathscr{B}$. Parallel lines represent the same lines via modulo arithmetic.}\\
\footnotesize{$\bullet$ Solid black lines represent the part of the tangent line connecting two vertices.}

\end{figure}

Similar to the triangle case, we also have a version of Lemma \ref{lem:h2_ss_0} to detect pairs $(r,s)$ that satisfy $H_{3,j}(r,s) = 0$ but $r = s$. 
\begin{lemma}
    \label{lem:h3_ss_0}
    Suppose char $\F_q \neq 2$ and let $H_{3,j}$ be the form of $H_3$ under $\Pen_j$. $H_{3,j}(s,s) = 0$ if and only if the matrix representation, $C_j(s)$, of the element corresponding to $s \in \Prj^1(\F_q)$ in pencil $\Pen_j$ is singular.
\end{lemma}

\begin{proof}
    Let $A = B = C_j(s)$, then $\sqrt{det(tA+B)} = \sqrt{(t+1)^3 \det(C_j(s))}$ which yields $H_{3,j}(s,s) = \frac{-1}{16}\sqrt{\det(C_j(s))}$ which is zero if and only if $A = C_j(s)$ is singular.
\end{proof}

Adapting a similar notation from the triangle case, we denote the set of valid pairs that correspond to $4$-Poncelet pairs as $\Gamma_j(H_3) = \{(r,s)\in \Phi_j \, | \, H_{3,j}(r,s) = 0 \}$, which we decompose into three disjoint sets

$\Gamma_j(H_3(\cdot,\cdot)) = \{(r,s)\in \Gamma_j(H_3) | \, r,s \in \F_q \}$,

$\Gamma_j(H_3(\cdot,\infty)) = \{(r,s)\in \Gamma_j(H_3) | \, r \in \F_q , \,  s = \infty \}$, and

$\Gamma_j(H_3(\infty,\cdot)) = \{(r,s)\in \Gamma_j(H_3) | \, r = \infty, \, s \in \F_q \}$.

Our approach to count $4$-Poncelet pairs is to treat $H_{3,j}(r,s)$ as a polynomial in $\F_q[s] [r]$ where we can see from Table \ref{tab:h3_summary} that it is at most cubic in $r$ for all $j$. We are going to show that for $j\neq 18$, $H_{3,j}(r,s)$ can be expressed as a product of factors $f_i(r,s) \in \F_q[s] [r]$ that are either linear or quadratic in $r$. 

\begin{proposition}
    \label{prop:n4_case_3_16}
    For $j \in \{3,16\}$, $\abs{\Gamma_j(H_3)} = 3q - 6$ and $\abs{\Gamma_j(H_3) \cap \Psi_j} = 3q-9$.
\end{proposition}

\begin{proof}
    We can show that $H_{3,j}(r,s) = f_1(r,s)f_2(r,s)f_3(r,s)$ where each factor $f_i(r,s)$ is linear. Since the product is zero if at least one of the factors are zero, we let $A_i = \{(r,s)\in \Phi_j \, | \, f_i(r,s) = 0 \}, i =1,2,3$ and the union $\bigcup_{i=1}^3 A_i$ contains all pairs $(r,s)$ such that $H_3(r,s) = 0$ and its cardinality can be obtained by Inclusion-Exclusion principle.
    
    For $\Pen_3$, we immediately see from Table \ref{tab:h3_summary} that we can let
    \begin{equation*}
        f_1(r,s) = r-s^2, \, f_2(r,s) = (2s-1)r - s^2, \, f_3(r,s) = r + (s^2-2s).
    \end{equation*}
    
    Now, $f_1(r,s) = 0 \iff r = s^2$ which gives us $q$ pairs since we can choose any $s \in \F_q$. For the second factor, we consider two cases. If $2s-1 = 0$ then $s = \frac{1}{2}$ which will make $f_2(r,s) = \frac{-1}{4} \neq 0$ which does not contribute any pair. If $2s-1 \neq 0$, then $f_2(r,s) = 0 \iff r = \frac{s^2}{2s-1}$ which gives us $q-1$ pairs since we cannot choose $s = \frac{1}{2}$. Finally, $f_3(r,s) = 0 \iff r = 2s-s^2$ which gives us again $q$ pairs. Going to the pairwise relations, $f_1(r,s) = f_2(r,s) = 0 \iff r = s^2 = \frac{s^2}{2s-1}$. Note that from the argument above, $s \neq \frac{1}{2}$ since $f_2(r,s) = 0$. Thus, we can solve the quadratic equation $s^2 = \frac{s^2}{2s-1}$ which is equivalent to $2s^2(s-1)=0$. Hence, we only have $2$ choices for this pair, which are $s=0$ or $s=1$. This will be the same idea for $f_1(r,s) = f_3(r,s) = 0$ and $f_2(r,s) = f_3(r,s) = 0$ which will both yield the same solution of $s=0$ and $s=1$. This also immediately gives us that $f_1(r,s) = f_2(r,s) = f_3(r,s) = 0 \iff s \in \{0,1\}$. Finally, we use Lemma \ref{lem:h3_ss_0} to remove unwanted pairs $(0,0)$ and $(1,1)$. 
    
    Notice that $H_3(\infty,s) = 2s-1$ and is zero if and only if $s=\frac{1}{2}$. Thus, the pair $\left(\infty,\frac{1}{2}\right)$ is the only valid pair in $\Gamma_3(H_3(\infty,\cdot))$. The set $\Gamma_3(H_3(\cdot,\infty))$ is empty since we don't consider the case where $C_j(s)$ is singular.
    
    Thus, for $j=3$, we obtain the count
    \begin{align*}
        \abs{\Gamma_j(H_3)} &= \abs{\Gamma_j(H_3(\cdot,\cdot))} + \underbrace{\abs{\Gamma_j(H_3(\cdot,\infty))}}_{=0} + \underbrace{\abs{\Gamma_j(H_3(\infty,\cdot))}}_{=1} \\
        &= \underbrace{q + (q-1) + q  - 2-2-2+2}_{\text{Inclusion-Exclusion}} - \underbrace{2}_{\text{Lemma } \ref{lem:h3_ss_0}} + 1   = 3q-6.
    \end{align*}
    
    Restricting to the smooth conics, we need to find the valid pairs in $\Gamma_j(H_3)$ that are not smooth. We already have the pair $\left(\infty,\frac{1}{2}\right)$ from $r = \infty$. Looking at the other singular elements, if $r=0$, then we obtain $H_3(0,s) = s^5 (s-2)$ which gives us one valid pair $(0,2)$. Finally, for $r=1$, $H_3(1,s) = (s-1)^5(s+1)$ giving us another valid pair $(1,-1)$. In total, we have $3$ valid pairs corresponding to $4$-Poncelet pairs that are not smooth pairs. Hence, for $j=3$,
    \begin{align*}
        \abs{\Gamma_j(H_3) \cap \Psi_j} &= \abs{\Gamma_j(H_3)} - 3 = 3q-9.
    \end{align*}
    
    A similar structure occurs for $j = 16$ by observing that we can let the first linear factor of $H_{3,16}(r,s)$ as
    \begin{equation*}
        f_1(r,s) = (2s-1)r - \pr{s^2+\frac{u^2-1}{4}}    
    \end{equation*}
    and the quadratic factor can be expressed as a product of two linear polynomials in $\F_q[s] [r]$, which we denote by
    \begin{equation*}
        f_2(r,s) = ur - \pr{us+s^2-s-\frac{u^2-1}{4}}, f_3(r,s) =  ur - \pr{us-s^2+s+\frac{u^2-1}{4}}
    \end{equation*}
    where $u = \sqrt{(1-4d)(1-4e)} \in \F_q \setminus \{0\}$ since, by assumptions in Table \ref{tab:pencil_class_auto}, $1-4d$ and $1-4e$ are both non-squares, which means their product must be a non-zero square.
    
    In this case, $f_1(r,s)$ will give one pair for each $s \neq \frac{1}{2}$, giving us a total of $q-1$ pairs while the other factors will contribute $q$ pairs. For the pairwise contributions, we have $f_i(r,s) = f_j(r,s) = 0$ for $i \neq j$ if and only if $s \in \{\eta_{+},\eta_{-}\}$ where $\eta_{+} = \frac{1+u}{2}$ and $\eta_{-}=\frac{1-u}{2}$ are the elements in $\F_q$ that corresponds to singular conics in pencil $\Pen_{16}$ as shown in Table \ref{tab:pencil_sing_base}. We also get one contribution for the pair $\left(\infty,\frac{1}{2}\right)$ and no contribution for $s = \infty$. Altogether, we get the same number of valid pairs that are $4$-Poncelet pairs as in the case of $j = 3$.
    
    Lastly, we have $H_{3,16}(\eta_{+},s) = 0$ if and only if $s \in \left\{ \eta_{+}, \frac{1-3u}{2} \right\}$, and $H_{3,16}(\eta_{-},s) = 0$ if and only if $s \in \left\{ \eta_{-}, \frac{1+3u}{2} \right\}$ which gives us the 3 valid pairs, $\left(\eta_{+},\frac{1-3u}{2}\right), \left(\eta_{-},\frac{1+3u}{2}\right), \left(\infty,\frac{1}{2}\right) \in \Gamma_j(H_3)$, that are not smooth pairs which again implies that we have the same number of smooth pairs that are $4$-Poncelet pairs as in the case of $j = 3$.

\end{proof}

\begin{proposition}
    \label{prop:n4_case_14_19}
    For $j \in \{14,19\}$, $\abs{\Gamma_j(H_3)} = q$ and $\abs{\Gamma_j(H_3) \cap \Psi_j} = q-1$.
\end{proposition}

\begin{proof}
    For this case, we have $H_3(r,s) = f_1(r,s)f_2(r,s)$ where $f_1(r,s)$ is linear while $f_2(r,s)$ is quadratic polynomial that does not have any roots for all $s \in \F_q$ and hence, $H_{3,j}(r,s) = 0$ if and only if the linear factor is $f_1(r,s) = 0$. 
    
    First, let us show that the linear factors are zero for $q-1$ pairs. For $j = 14$, the linear factor is $0$ if and only if $r = \frac{es^2 - 1}{2es-1}$ which gives us one pair $(r,s)$ for each $s \in \F_q \setminus \left\{\frac{1}{2e} \right\}$. Similarly, for $j = 19$, the linear factor is $0$ if and only if $r = \frac{\nu s^2 - \sigma^2}{2\nu s + \rho}$ which gives us one pair $(r,s)$ for each $s \in \F_q \setminus \left\{\frac{-\rho}{2\nu} \right\}$. 
    
    Going to the quadratic factor, for $j=14$, $D(f_2(r,s)) = 4(es^2-s+1)^2(1-4e)$. By the assumption in Table \ref{tab:pencil_class_auto} that $T^2 + T + e$ is irreducible, we have that its discriminant $1-4e$ is non-square. Since the polynomial $es^2-s+1$ shares the same discriminant, it is also irreducible and will not have any roots in $\F_q$. Hence, the factor $4(es^2-s+1)^2$ is always a non-zero square and $D(f_2(r,s))$ is always a non-square for any $s \in \F_q$. Similarly, for $j=19$, $D(f_2(r,s)) = 4(\nu s^2 + \rho s+ \sigma^2)^2(\rho^2-4\nu\sigma^2)$ and using the assumptions for $\nu, \rho$ and $\sigma$ in Table \ref{tab:pencil_class_auto}, we arrive with the same conclusion as in the case of $j=14$.
    
    Finally, the coefficient of $r$ in the linear factor turns out to be $H_{3,j}(\infty,s)$ which gives us the valid pairs $\left(\infty,\frac{1}{2e}\right)$ for $j = 14$ and $\left(\infty,\frac{-\rho}{2\nu}\right)$ for $j = 19$. Adding all these contributions, we obtain
    \begin{align*}
        \abs{\Gamma_j(H_2)} &= \abs{\Gamma_j(H_2(\cdot,\cdot))} + \underbrace{\abs{\Gamma_j(H_2(\cdot,\infty))}}_{=0} + \underbrace{\abs{\Gamma_j(H_2(\infty,\cdot))}}_{=1} \\
        &= (q-1) + 1  = q.
    \end{align*}
    
    Since the valid pairs in $\Gamma_j(H_3)$ that are not smooth pairs are just those in $\Gamma_j(H_2(\infty,\cdot))$
    \begin{align*}
        \abs{\Gamma_j(H_3) \cap \Psi_j} &= \abs{\Gamma_j(H_3)} - 1 = q-1.
    \end{align*}
\end{proof}

For the case of pencil $\Pen_{18}$, $H_3$ from Cayley's condition becomes unwieldy, and we instead opt for a geometric argument to justify that we cannot construct a Poncelet tetragon for any pair in this pencil. First, we need this lemma relating base points and a Poncelet polygon with overlapping vertices.

\begin{lemma}
    \label{lem:poly_degen}
    Let $P_k$ be the $k$th vertex of a Poncelet $n$-gon inscribed in $\overline{\A}$ and circumscribed about $\overline{\B}$ with $P_0 = P_{n}$ and $P_1 = P_{n+1}$. Then, $P_k \in \overline{\A} \cap \overline{\B}$ if and only if $P_{k-1} = P_{k+1}$. Moreover, $P_k \in \overline{\A} \cap \overline{\B}$ is an intersection of multiplicity greater than 1 if and only if $P_{k-1} = P_{k} = P_{k+1}$.
\end{lemma}

\begin{proof}
Let $\tau_k$ be the tangent line to $\B$ passing through $P_k$ and $P_{k+1}$. $P_k \in \overline{\A}$ since it is a vertex of our Poncelet polygon. Notice that $P_k \in \overline{\B}$ if and only if the polar line to $\B$ with respect to $P_k$ coincides with the tangent line to $\overline{\B}$ at $P_k$. This is equivalent to having the $\tau_k = \tau_{k-1}$, which happens if and only if $P_{k-1} = P_{k+1}$.

Now, in the case where $P_{k-1} = P_{k} = P_{k+1}$, $P_k = P_{k+1}$ would imply that $\tau_k$ intersects $\A$ at a point of multiplicity 2. Hence, $\tau_k$ is a common tangent of $\A$ and $\B$ at $P_k$, which will only happen if $P_k \in \overline{\A} \cap \overline{\B}$ is an intersection of multiplicity 2 or above. Similar argument holds for $P_{k-1} = P_{k}$.
\end{proof}

\begin{proposition}
    \label{prop:n4_case_18}
    $\abs{\Gamma_{18}(H_3)} = 0$.
\end{proposition}

\begin{proof}
    Suppose $(\A,\B)$ is a $4$-Poncelet pair in pencil $\Pen_{18}$. By Poncelet's theorem, we can choose the starting vertex $P_1 = [0:0:1] \in \A \cap \B$ and by Lemma \ref{lem:poly_degen}, we have $P_2 = P_4$ and $P_3 \in \overline{\A} \cap \overline{\B}$. However, Table \ref{tab:pencil_sing_base} implies that $P_1$ is the only base point on  $\Prj^2(\F_q)$ while, according to the irreducibility assumption of $T^3 + bT^2 + cT + 1$ in Table \ref{tab:pencil_class_auto}, the homogeneous coordinates of the remaining base points cannot be obtained from a solution of a quadratic equation in $\F_q[T]$. 
    Thus, the remaining base points cannot be obtained as an intersection of any line and conic in $\Prj^2(\F_q)$, forcing $P_3 = P_1$. Applying Lemma \ref{lem:poly_degen} and the previous argument, we conclude that $P_2$ is another base point on the base plane. Hence, we have $P_1 = P_2 = P_3$, which contradicts Lemma \ref{lem:poly_degen} since $P_2$ is a base point with intersection multiplicity 1.
\end{proof}

\begin{proposition}
    \label{prop:n4_case_4_6_15_17}
    For $j \in \{4,6,15,17\}$, $\abs{\Gamma_j(H_3)} = \abs{\Gamma_j(H_3) \cap \Psi_j} = q-1$.
\end{proposition}

\begin{proof}
    All of these share the same form of $H_{3,j}(r,s)$, which is degree 1 in both $r$ and $s$ with the roles of $r$ and $s$ reversed for $j=17$. Choosing $s \in \F_q \setminus \{0\}$, $H_{3,j}(r,s) = 0$ if and only if $r = \frac{s}{2}$ for $j \neq 17$ and $r = 2s$ for $j = 17$. Thus, there is only one value of $r$ for each $s$, giving us a total of $q-1$ pairs. We will not count the pairs of the form $(\infty,s)$ for $j \neq 17$ and $(0,s)$ for $j=17$ since they correspond to the irregular pairs.
\end{proof}

\begin{proposition}
    \label{prop:n4_case_5_8}
    For $j \in \{5,8\}$, $\abs{\Gamma_j(H_3)} = 0$.
\end{proposition}

\begin{proof}
    All valid pairs $(r,s)$ will have $H_{2,j}(r,s) \neq 0$. Pairs of the form $(\infty,s)$ correspond to irregular pairs, which we do not count.
\end{proof}

Similar to the triangle case, we summarize all the counts for the tetragon case in Theorem \ref{thm:n4_counts}.

\begin{theorem}
\label{thm:n4_counts}

The number of pairs corresponding to $4$-Poncelet pairs among the valid pairs and smooth pairs, for each pencil, is summarized in Table \ref{tab:n4_count}.

\begin{table}[!htb]
\centering
\renewcommand{\arraystretch}{2}%
\caption{Count of $4$-Poncelet pairs for each pencil in $\Prj^2(\F_q)$}
\begin{tabular}{l|c|c}
\hline
$\Pen_j$  & Smooth Pairs & Valid Pairs\\
\hline
$\Pen_{3}, \Pen_{16}$         & $3(q-3)$ &  $3(q-2)$  \\
\hline
$\Pen_{4}, \Pen_{6}, \Pen_{15}, \Pen_{17}$      &   $q-1$           &  $q-1$           \\
\hline
$\Pen_{5}, \Pen_{8}$     &  $0$            &    $0$         \\
\hline
$\Pen_{14}, \Pen_{19}$      &  $q-1$            &  $q$           \\
\hline
$\Pen_{18}$        &   $0$           &    $0$  \\
\hline
\end{tabular}
\label{tab:n4_count}
\end{table}

\end{theorem}

\begin{proof}
    See Propositions \ref{prop:n4_case_3_16}, \ref{prop:n4_case_14_19}, \ref{prop:n4_case_18}, \ref{prop:n4_case_4_6_15_17}, and \ref{prop:n4_case_5_8}.
\end{proof}

Using Theorem \ref{thm:n4_counts} and following the conventions stated at the beginning of Section \ref{sec:prob}, we also get the probabilities for the tetragon case for each pencil in Corollary \ref{cor:n4_prob}.

\begin{corollary}
    \label{cor:n4_prob}
    The probabilities of obtaining a valid pair or a smooth pair corresponding to a $4$\nobreakdash-Poncelet pair for each pencil $\Pen_j$ under our assumed probability space are summarized in Table \ref{tab:n4_prob}.

    \begin{table}[!htb]
    \centering
    \renewcommand{\arraystretch}{2}%
    \caption{Probabilities of obtaining a $4$-Poncelet pair for each pencil in $\Prj^2(\F_q)$}
    \begin{tabular}{l|c|c}
    \hline
    $\Pen_j$ & Smooth Pairs & Valid Pairs \\
    \hline
    $\Pen_{3}, \Pen_{16}$       &  $\frac{3}{q-2}$ for $q>3^{\dag}$ & $\frac{3}{q}$  \\
    \hline
    $\Pen_{4}, \Pen_{15}$      &   $\frac{1}{q-2}$           &  $\frac{1}{q-1}$           \\
    \hline
    $\Pen_{5}, \Pen_{8}$       &  $0$            &    $0$         \\
    \hline
    $\Pen_{6}, \Pen_{17}$      &   $\frac{1}{q-2}$           &  $\frac{1}{q-2}$           \\
    \hline
    $\Pen_{14}, \Pen_{19}$     &  $\frac{1}{q}$            &  $\frac{1}{q}$           \\
    \hline
    $\Pen_{18}$        &   $0$           &    $0$  \\
    \hline
    \end{tabular}
    \label{tab:n4_prob}
    \flushleft
\footnotesize{${}^{\dag}$: The sample space is empty for $q=3$}
    \end{table}
    
\end{corollary}

\begin{remark}
    The computed probabilities of obtaining a $4$-Poncelet pair for pencils with transversally intersecting elements are consistent with the asymptotic results in \cite{wan2025pnc}.
\end{remark}

\clearpage

\section{Final remarks}
\label{sec:conclusion}

A natural direction for further research is to compute the probabilities of obtaining an $n$-Poncelet pair in a fixed pencil in $\Prj^2(\F_q)$ for $n > 4$. This direction can be thought of as fixing the number of sides $n$ and studying how the probability varies with the order $q$. 

It will also be interesting to ask a symmetric question on how the number of sides $n$ is distributed for a fixed order $q$, which has been studied in \cite{hunkus2020poncelet} for smooth pairs in pencil $\Pen_{17}$.
We performed a computational experiment in Python with the aid of Galois package \cite{hos2020galpack} for finite field computations where, for every valid pair in $\Prj^2(\F_q)$ with $q \leq 19$, we computed $n$, the number of sides of the Poncelet polygon that can be constructed from that pair. 
Based on a pattern involving non-smooth valid pairs that intersect non-transversally, we state the following:

\begin{conjecture}
    For pencils $\Pen_{4}$ and $\Pen_{15}$ in $\Prj^2(\F_q)$ and $k = 2 \, \chr \F_q$, the pair $(0,s)$, where $s \in \F_q\setminus\{0\}$, corresponds to a $k$-Poncelet pair.
\end{conjecture}

While in this paper we restrict the analysis to the probability of choosing pairs of conics within a fixed pencil, a natural extension would be to consider the probability when choosing two arbitrary conics in the projective plane.
By utilizing the law of total probabilities, our results can be used to compute the probability of obtaining a $3$\nobreakdash-Poncelet pair or a $4$\nobreakdash-Poncelet pair among all possible valid pairs of conics, with a conveniently chosen probability measure on that set, which would also require computing probabilities for a type of pencil determined by a randomly chosen pair of conics.
Another way is used in \cite{chi2017ptc} and \cite{wan2025pnc}, where they computed the total probabilities in a two-stage process: first, randomly selecting a pencil, and then from the chosen pencil, randomly choosing a pair of smooth conics.
In \cite{chi2017ptc}, the probability of selecting a pencil was defined to be proportional to the number of smooth pairs in that pencil, while in \cite{wan2025pnc} such probability is proportional to obtaining the base point configuration of that pencil if the base points are randomly selected.
	
It is interesting to note that our results can also be interpreted in terms of solids in the $5$-dimensional projective space over $\F_q$, with $q$ odd. 
Namely, since the conics in the projective plane can be naturally identified with the hyperplanes in the $5$-dimensional projective space, the present work computes the probability that two such hyperplanes satisfying certain conditions satisfy the Poncelet property for triangles and tetragons
whose sides are conics on the Veronese surface.
See \cite{al2025} for the details.

\section*{Acknowledgments}
The research of M.~R.~was supported by the Australian Research Council, Discovery Project 190101838 \emph{Billiards within quadrics and beyond}, the Serbian Ministry of Science, Technological Development and Innovation and the Science Fund of Serbia grant IntegraRS.

R.~R.~would like to acknowledge the support from the Faculty of Science Research Tuition Fee Scholarship and the Postgraduate Research Scholarship in Mathematics and Statistics provided by the University of Sydney. 

The authors are grateful to the referees for their careful reading and insightful comments and suggestions.


\begin{appendices}

\section{Classification of Pencils of Conics}
\label{sec:dicksonclass}

In this appendix, we present Dickson's classification \cite{dic1908penauto}, which is a complete classification of all pencils of conics in $\Prj^2(\F_q)$ that are unique up to projective automorphism.
We recreate this classification in Table \ref{tab:pencil_class_auto} as found in \cite{hir1998projfinite} and denote the pencil in the $j$th row as $\Pen_j$. 
Before explaining that table, we recall some essential facts about conics and their pencils.
\begin{table}[!htb]
\centering
\renewcommand{\arraystretch}{1.5}%
\caption{Dickson classification \cite{dic1908penauto,hir1998projfinite}}
\begin{tabular}{c|c|c|c|c|c|c|c}
\cline{1-8}
\multicolumn{1}{c|}{\multirow{2}{*}{$\Pen_j$}} & \multicolumn{1}{c|}{\multirow{2}{*}{Base Points}} & \multicolumn{4}{c|}{Count per Type of Conic}                                                                          & \multicolumn{2}{c}{Generators}                     \\ \cline{3-8} 
\multicolumn{1}{c|}{} & \multicolumn{1}{c|}{}                             & \multicolumn{1}{c|}{Smooth} & \multicolumn{1}{c|}{Two Lines} & \multicolumn{1}{c|}{Point} & \multicolumn{1}{c|}{Line} & \multicolumn{1}{c|}{$\A_j$} & \multicolumn{1}{c}{$\B_j$} \\ \hline
    $\Pen_{1}$   &      & 0             & $q$                 & 0                 & 1           & $x^2$                     & $xy$ \\
    \hline
    $\Pen_{2}$   &       & 0             & $q+1$               & 0                 & 0           & $xy$                      & $xz$ \\
    \hline
    $\Pen_{3}$   & $(1,1,1,1)1$      & $q-2$           & 3                 & 0                 & 0           & $xy$                      & $xz+yz+z^2$ \\
    \hline
    $\Pen_{4}$   & $(1,1,2)1$        & $q-1$           & 2                 & 0                 & 0           & $xy$                      & $xz+z^2$ \\
    \hline
    $\Pen_{5}$   & $(1,3)1$          & $q$             & 1                 & 0                 & 0           & $xy$                      & $xz+y^2$ \\
    \hline
    $\Pen_{6}$   & $(2,2)1$          & $q-1$           & 1                 & 0                 & 1           & $xy$                      & $z^2$ \\
    \hline
    $\Pen_{7} ^{**}$   & $(4)1$      & 0             & 0                 & 0                 & $q+1$       & $x^2$                     & $y^2$ \\
    \hline
    $\Pen_{8}$   & $(4)1$            & $q$             & 0                 & 0                 & 1             & $x^2$                     & $xy+z^2$ \\
    \hline
    $\Pen_{9} ^{**}$   & $(4)1$      & 0             & $\frac{q}{2}$     & $\frac{q}{2}$     & 1             & $x^2$                     & $xy+y^2$ \\
    \hline
    $\Pen_{10} ^{\#}$  & $(4)1 $      & 0             & $\frac{q-1}{2}$   & $\frac{q-1}{2}$   & 2             & $xy$                      & $x^2-y^2$ \\
    \hline
    $\Pen_{11} ^{\#}$  & $(4)1 $      & 0             & $\frac{q+1}{2}$   & $\frac{q+1}{2}$   & 0             & $xy$                      & $x^2-\nu y^2 {}^{\dag}$ \\
    \hline
    $\Pen_{12} ^{\#\#}$  & $(4)1 $    & 0             & $\frac{q-1}{2}$   & $\frac{q-1}{2}$   & 2             & $xy$                      & $x^2+y^2$  \\
    \hline
    $\Pen_{13} ^{\#\#}$  & $(4)1 $    & 0             & $\frac{q+1}{2}$   & $\frac{q+1}{2}$   & 0             & $xy$                      & $x^2-y^2$  \\
    \hline
    $\Pen_{14}$  & $(1,1)1,(1)2$     & $q$             & 1                 & 0                 & 0             & $xy$                      & $xz+y^2+yz+e z^2 {}^{\ddag}$ \\
    \hline
    $\Pen_{15}$  & $(2)1,(1)2$       & $q-1$           & 1                 & 1                 & 0             & $xy$                      & $y^2+yz+e z^2 {}^{\ddag}$ \\
    \hline
    $\Pen_{16}$  & $(1,1)2$          & $q-2$           & 1                 & 2                 & 0             & $xy$                      & $e x^2 + xz + d y^2 + yz + z^2  {}^{\ddag}$ \\
    \hline
    $\Pen_{17}$  & $(2)2$            & $q-1$           & 0                 & 1                 & 1             & $x^2$                     & $y^2+yz+e z^2 {}^{\ddag}$ \\
    \hline
    $\Pen_{18}$  & $(1)1,(1)3$       & $q+1$           & 0                 & 0                 & 0             & $-xz + y^2$                & $x^2+cxy+by^2+yz {}^{\ddag}$ \\
    \hline
    $\Pen_{19} ^{*}$  & $(1)4 $       & $q$             & 0                 & 1                 & 0             & $x^2 - \nu y^2 {}^{\dag}$  & $2 \sigma xy -\rho y^2 + z^2 {}^{\dag}$  \\
    \hline
    $\Pen_{20} ^{**}$  & $(1)4 $      & $q$             & 0                 & 1                 & 0             & $x^2+xy+e y^2 {}^{\ddag}$  & $f xz + y^2 + z^2 {}^{\ddag}$ \\
    \hline
\end{tabular}

\label{tab:pencil_class_auto}
\flushleft
\footnotesize{$^\dag$ $\nu$ and $\rho^2 - 4 \nu \sigma^2$ are non-squares.}
\footnotesize{$^\ddag$ $T^3+bT^2+cT+1,\, T^2+T+d,\, T^2+T+e,\,$ and $ T^2 + e f^2 T + f^2$ are irreducible in $\F_q[T]$.}

\footnotesize{$^{*}$  $q$ must be odd.}
\footnotesize{$^{**}$  $q$ must be even.}
\footnotesize{$^{\#}$  $q \equiv 1 \mod 4$.}
\footnotesize{$^{\#\#}$  $q \equiv -1 \mod 4$.}
\end{table}

In the projective plane $\Prj^2(\F_q)$, there are \emph{four types of conics}  -- each conic is one of the following:
\begin{itemize}	
\item a smooth conic;
\item the union of two distinct lines;
\item a single line;
\item a point.
\end{itemize}	
Note that a conic which coincides with a line is treated as a double line, while a conic coinciding with a point will in fact be, in the projective plane over the algebraic closure $\overline{\F}_q$, the union of two conjugate lines which intersect at that point. 
We refer to the non-smooth conics as \emph{singular conics}.

Recall that the \emph{base points} of a given pencil are common points of all conics in the pencil.
Since pencils are liner families, the set of base points coincides with the intersection of any two elements of the pencil.
The configuration of the base points for each pencil is described in the second column of Table \ref{tab:pencil_class_auto}.

Aside from $\Pen_1$ and $\Pen_2$, the conics in the remaining pencils have no common components.
Thus, according to B\'ezout's theorem, each of those pencils has in $\Prj^2(\overline{\F_q})$ four base points counting multiplicity.
For example, the configuration $(1,1,1,1)1$ represents four base points of multiplicity $1$ in $\Prj^2(\F_q)$, while $(2)1$ represents a base point in $\F_q$ with multiplicity $2$, and $(1)2$ represents two base points that are conjugate in $\Prj^2(\F_{q^2})$. 
More generally, $(j)k$ represents $k$ conjugate base points of multiplicity $j$ in the projective plane over $\F_{q^k}$.

Two conics intersect transversally if they intersect at four distinct points in $\Prj^2(\overline{\F_q})$; that is, each base point must have intersection multiplicity 1. This can be any pair of conics in pencils $\Pen_{3}, \Pen_{14}, \Pen_{16}, \Pen_{18}$ and $\Pen_{19}$, which are the pencils considered in \cite{chi2017ptc}. In this paper, we also consider pencils with non-transversally intersecting elements, which refers to $\Pen_{4}, \Pen_{5}, \Pen_{6}, \Pen_{8}, \Pen_{15}$ and $\Pen_{17}$.

Every pencil of conics in $\Prj^2(\F_q)$ contains $q+1$ conics and the distribution of each type of conics is given in the third column.

For pencils considered in this paper, we compute the homogeneous coordinates of the base points together with the index $\eta \in \Prj^1(\F_q)$ corresponding to singular elements in that pencil. This is summarized in Table \ref{tab:pencil_sing_base}.

\begin{table}[!htb]
\centering
\renewcommand{\arraystretch}{1.5}%
\caption{Singular elements and base points of each pencil}
\begin{tabular}{c|c|l}
\hline
$\Pen_j$ & Singular $\eta {}^{*}$        & Base Points (Intersection Multiplicity)  \\
\hline
$\Pen_{3}$      & $0, 1 , \infty$        &  $ [-1:0:1], [0:-1:1], [0:1:0], [1:0:0]$            \\
\hline
$\Pen_{4}$      & $0, \infty $           &  $ [-1:0:1], [1:0:0], [0:1:0] (2) $            \\
\hline
$\Pen_{5}$      & $\infty$               &  $ [1:0:0], [0:0:1] (3)$            \\
\hline
$\Pen_{6}$      & $0, \infty $           &  $[0:1:0] (2), [1:0:0] (2)$            \\
\hline
$\Pen_{8}$      & $\infty$               &  $[0:1:0] (4)$            \\
\hline
$\Pen_{14}$     & $\infty$               &  $[-e:0:1], [1:0:0], [0:\mu_1 {}^{\dag}:1], [0:\mu_2 {}^{\dag}:1]$            \\
\hline
$\Pen_{15}$     & $0, \infty $           &  $[1:0:0] (2), [0:\mu_1 {}^{\dag}:1], [0:\mu_2 {}^{\dag}:1] $            \\
\hline
$\Pen_{16}$     & $\eta_{+} {}^{**}, \eta_{-} {}^{**} , \infty$  &  $[0:1:\mu_3 {}^{\ddag}], [0:1:\mu_4 {}^{\ddag}], [1:0:\mu_1 {}^{\dag}], [1:0:\mu_2 {}^{\dag}]$          \\
\hline
$\Pen_{17}$     &  $0, \infty $          &  $[0:\mu_1 {}^{\dag}:1]  (2), [0:\mu_2 {}^{\dag}:1] (2)$            \\
\hline
$\Pen_{18}$     &  None                  &  $[0:0:1], [1:\mu_5 {}^{\#}:\mu^2_5 {}^{\#}], [1:\mu_6 {}^{\#}:\mu^2_6 {}^{\#}], [1:\mu_7 {}^{\#}:\mu^2_7 {}^{\#}]$            \\
\hline
$\Pen_{19}$     &  $\infty$              & $[-\sqrt{\nu}:1:\zeta_1 {}^{\#\#}], [-\sqrt{\nu}:1:-\zeta_1 {}^{\#\#}], [\sqrt{\nu}:1:\zeta_2 {}^{\#\#}], [\sqrt{\nu}:1:-\zeta_2 {}^{\#\#}]$ \\
\hline
\end{tabular}
\label{tab:pencil_sing_base}

\flushleft

\footnotesize{$^*$ $\eta$ that correspond to singular conics in $\Pen_j$; that is, $\det C_j(\eta) = 0$}

\footnotesize{$^{**}$ $\eta_{+} = \dfrac{1+\sqrt{(1-4d)(1-4e)}}{2}, \eta_{-} = \dfrac{1-\sqrt{(1-4d)(1-4e)}}{2}$}

\footnotesize{$^\dag$ $\mu_1$,$\mu_2$ are roots of $T^2 + T + e$ in $\overline{\F_q}$}
\footnotesize{$^\ddag$ $\mu_3$,$\mu_4$ are roots of $T^2 + T + d$ in $\overline{\F_q}$}
\footnotesize{$^\#$ $\mu_5$,$\mu_6$,$\mu_7$ are roots of $T^3 + bT + cT + 1$ in $\overline{\F_q}$}

\footnotesize{$^{\#\#}$ $\zeta_1 = \sqrt{\rho+2\sigma\sqrt{\nu}}, \, \zeta_2 = \sqrt{\rho - 2\sigma\sqrt{\nu}}$}

\end{table} 

\end{appendices}

\nocite{*}
\printbibliography

\end{document}